\documentclass{amsart}

\usepackage{amsmath,amssymb,amsthm} 
\usepackage[all]{xy}
\usepackage{extarrows}
\usepackage{graphicx}
\usepackage{bm}
\usepackage{comment, url}
\usepackage[normalem]{ulem}

\usepackage{multicol}

\newtheorem{thm}{Theorem}[section]
\newtheorem{lem}[thm]{Lemma}
\newtheorem{cor}[thm]{Corollary}
\newtheorem{prop}[thm]{Proposition}  

\theoremstyle{remark}

\newtheorem{nota}{Notation and terminology\!}    

\theoremstyle{definition}
\newtheorem{defn}[thm]{Definition}
\newtheorem{rem}[thm]{Remark}

\newtheorem{conj}[thm]{Conjecture}
\newtheorem{q}[thm]{Question} 

\newtheorem{def/prop}[thm]{Definition/Proposition}

\numberwithin{equation}{section}

\usepackage[usenames]{color}

\newcommand{\red}{}   \newcommand{\violet}{}  \newcommand{\magenta}{} \newcommand{\cyan}{} \newcommand{\fgreen}{} \newcommand{\orange}{}

\newcommand{\rs}[1]{\textcolor{red}{\sout{#1}}}

\def\det{\mathop{\mathrm{det}}\nolimits}
\def\Im{\mathop{\mathrm{Im}}\nolimits}
\def\Ker{\mathop{\mathrm{Ker}}\nolimits}

\def\Hom{\mathop{\mathrm{Hom}}\nolimits}

\def\Gal{\mathop{\mathrm{Gal}}\nolimits}
\def\Spec{\mathop{\mathrm{Spec}}\nolimits}

\def\mod{\mathop{\mathrm{mod}}\nolimits}
\def\Cl{\mathop{\mathrm{Cl}}\nolimits}

\def\Nr{\mathop{\mathrm{Nr}}\nolimits}

\newcommand{\mf}[1]{{\mathfrak{#1}}}

\newcommand{\bb}[1]{{\mathbb{#1}}}
\newcommand{\mca}[1]{{\mathcal{#1}}}

\newcommand{\To}{\longrightarrow}
\newcommand{\LR}{\longleftrightarrow}
\newcommand{\inj}{\hookrightarrow}
\newcommand{\surj}{\twoheadrightarrow}

\newcommand{\congto}{\overset{\cong}{\to}}

\newcommand{\imp}{\Longrightarrow}

\newcommand{\iffu}{\underset{\rm iff}{\iff}}

\newcommand{\N}{\bb{N}}
\newcommand{\Z}{\bb{Z}}

\newcommand{\Q}{\bb{Q}}

\newcommand{\R}{\bb{R}}
\newcommand{\C}{\bb{C}}
\newcommand{\F}{\bb{F}}
\newcommand{\p}{\mf{p}}
\renewcommand{\P}{\mf{P}}

\renewcommand{\(}{\left\{}
\renewcommand{\)}{\right\}}
\newcommand{\del}{\partial}

\newcommand{\ol}{\overline}
\newcommand{\ul}{\underline}

\newcommand{\ds}{\displaystyle}

\newcommand{\wt}[1]{{\widetilde{#1}}}
\newcommand{\wh}[1]{{\widehat{#1}}}

\DeclareMathOperator*{\restprod}%
 {\mathchoice{\ooalign{\ensuremath{\displaystyle\prod}\crcr\ensuremath{\displaystyle\coprod}}}%
             {\ooalign{\ensuremath{\textstyle\prod}\crcr\ensuremath{\textstyle\coprod}}}%
             {\ooalign{\ensuremath{\scriptstyle\prod}\crcr\ensuremath{\scriptstyle\coprod}}}%
             {\ooalign{\ensuremath{\scriptscriptstyle\prod}\crcr\ensuremath{\scriptscriptstyle\coprod}}}%
 }

\begin{document}



\title[Id\`elic 
 for 3-manifolds and very admissible links]{Id\`{e}lic class field theory for\\ 3-manifolds and very admissible links}

\author{Hirofumi Niibo and Jun Ueki}





\maketitle

\begin{abstract} 
We study a topological analogue of id\`elic class field theory for 3-manifolds, in the spirit of arithmetic topology. 
We firstly introduce the notion of a very admissible link $\mca{K}$ in a 3-manifold $M$, which plays a role analogous to the set of primes of a number field. 
For such a pair $(M,\mca{K})$, we introduce the notion of id\`eles and define the id\`ele class group. 
Then, getting \red{the} local class field theory for each knot in $\mca{K}$ together, 
we establish analogues of the global reciprocity law and the existence theorem of id\`elic class field theory. 
\end{abstract}

\footnote[0]{keywords: id\`ele, class field theory, 3-manifold, branched covering, arithmetic topology.}
\footnote[0]{Mathematics Subject Classification 2010: \fgreen{Primary 57M12, 11R37, Secondary 57M99}} 

\section{Introduction} 
\violet{In this paper, following analogies between 3-dimensional topology and number theory, we study a topological analogue of id\`elic class field theory for 3-manifolds. 
We establish topological analogues of Artin's global reciprocity law and the existence theorem \fgreen{of} 
 id\`elic class field theory, for any closed, oriented, connected 3-manifold.} 


Let $M$ be a closed, oriented, connected 3-manifold, equipped with a \emph{very admissible link} (knot set) $\mca{K}$, the notion introduced in \cite{Niibo1} as an analogue of the set of primes in a number ring. 
For this notion, we give a refined treatment in Section 2. 

For $(M,\mca{K})$, we also introduce the notion of a \emph{universal $\mca{K}$-branched cover} in Section 3, \red{which may be regarded as} an analogue of an algebraic closure of a number field. 
We fix it and restrict our argument to the branched covers which are obtained as its quotients. (It is equivalent to consider isomorphism classes of branched covers with base points.) 

\fgreen{We then introduce the \emph{id\`ele group} $I_{M,\mca{K}}$, the  \emph{principal id\`ele group} $P_{M,\mca{K}}$, and the \emph{id\`ele class group} $C_{M,\mca{K}}:=I_{M,\mca{K}}/P_{M,\mca{K}}$ in a functorial way.  
We define $I_{M,\mca{K}}$ as the restricted product of $H_1(\del V_K)$'s where $K$ runs through all the components of $\mca{K}$, following \cite{Niibo1}. 
Let $\Gal(X_L^{\rm ab}/X_L)$ denote the Galois group of the maximal abelian cover over the exterior $X_L=M-L$ of each finite sublink $L\subset \mca{K}$, and put $\Gal(M,\mca{K})^{\rm ab}:=\varprojlim_{L\subset \mca{K}}\Gal(X_L^{\rm ab}/X_L)$. Then there is a natural homomorphism $\wt{\rho}_{M,\mca{K}}:I_{M,\mca{K}}\to \Gal(M,\mca{K})^{\rm ab}$. 
In \cite{Niibo1}, $P_{M, \mca{K}}$ was defined as $\Ker \wt{\rho}_{M,\mca{K}}$. In this paper, we define it by the image of the natural homomorphism $\Delta:H_2(M,\mca{K})\to I_{M,\mca{K}}$ and prove that it coincides with that of \cite{Niibo1} 
(Theorem \ref{Pidele}). 
}

The first main result of 
\orange{our id\`elic class field theory} 
is the following analogue of Artin's global reciprocity law, 
which was \red{originally} proved 
in \cite{Niibo1} for the case over an integral homology 3-sphere. 
\\[3mm]
\noindent 
\textbf{Theorem \ref{global CFT for mfd}} (The global reciprocity law for 3-manifolds)\textbf{.} 
\emph{There is a canonical isomorphism called 
the global reciprocity map 
\[ \rho_{M,\mca{K}}: C_{M,\mca{K}}\congto \Gal(M,\mca{K})^{\rm ab}\]
such that} (i) \emph{for any finite abelian cover  $h:N \to M$ branched over a finite link in $\mca{K}$, $\rho_{M,\mca{K}}$ induces an isomorphism of finite abelian groups 
\[ C_{M,\mca{K}}/h_{*}(C_{N,h^{-1}(\mca{K})})\congto \Gal(h),\]
and}
(ii) \emph{$\rho_{M,\mca{K}}$ is compatible with the local theories.
}\\[-2mm]

In this paper, we also show an analogue of the existence theorem of id\`elic  class field theory, which gives a bijective correspondence between finite abelian covers of $M$ branched over finite links in $\mca{K}$ and certain subgroups of $C_{M,\mca{K}}$. 
For this purpose, we introduce certain topologies on $C_{M,\mca{K}}$, called the \emph{standard topology} and the \emph{norm topology}, following after the case of number fields (\cite{KKS2}, \cite{Neukirch}). 
Now our existence theorem is stated as follows:
\\[3mm]
\noindent 
\textbf{Theorem \ref{main thm}} (The existence theorem)\textbf{.} 
\emph{
The following correspondence 
\[(h:N\to M) \mapsto h_*(C_{N,h^{-1}(\mca{K})})\] 
gives a bijection between the set of (isomorphism classes of) finite abelian covers of $M$ branched over finite links $L$ in $\mca{K}$ and the set of open subgroups of finite indices of $C_{M,\mca{K}}$ with respect to the standard topology. 
Moreover, the latter set coincides with the set of open subgroups of $C_{M,\mca{K}}$ with respect to the norm topology.} \\[-2mm] 

Here are contents of this paper. 
In Section 2, we discuss very admissible links. We refine the definition in \cite{Niibo1}, prove the existence, and give some remarks. 
In Section 3, we briefly review the basic analogy between knots and primes, introduce the notion of a universal $\mca{K}$-branched cover, and discuss the role of base points.
In Section 4, we recall the id\`elic class field theory for number fields, whose analogue will be discussed in the later sections.  
In Section 5, 
 we introduce 
 \red{the} id\`ele group $I_{M,\mca{K}}$, 
\fgreen{the natural homomorphism $\Delta:H_2(M,\mca{K})\to I_{M,\mca{K}}$, the principal id\`ele group $\mca{P}_{M,\mca{K}}$,} 
and \red{the} id\`ele class group $C_{M,\mca{K}}$, together with the global reciprocity map $\rho_{M,\mca{K}}:C_{M,\mca{K}}\to \Gal(M,\mca{K})^{\rm ab}$ for a 3-manifold $M$ equipped with a \fgreen{(}very admissible\fgreen{)} link $\mca{K}$. 
Then we verify the global reciprocity law for 3-manifolds. 
In Section 6, we introduce the standard topology on the id\`ele class group $C_{M,\mca{K}}$, and prove the existence theorem for it.
In Section 7, we do the same for the norm topology. 
In Section 8, \magenta{we give some remarks on the norm residue symbols, the class field axiom, and an application to the genus theories.}\\ 

We note that id\`elic class field theory for 3-manifolds was \red{initially} studied by A.\ Sikora (\cite{Sikora2003}, \cite{Sikora-note}, \cite{Sikora-slide}). 
\fgreen{In the beginning of our study we were inspired by his work, 
although our approach presented in this paper is different from his.} 

\begin{nota}
For a manifold $X$, we simply denote by $H_{\fgreen{n}}(X)$ its \fgreen{$n$}-th homology group with coefficients in $\Z$. For a group $G$ and its subgroup $G_1$, we write $G_1< G$. 
\fgreen{A knot (resp. a link) in a 3-manifold $M$ means a topological embedding of $S^1$ (resp. $\sqcup S^1$) into $M$ or its image. 
A branched cover of a 3-manifold is branched over a link, and is a morphism of spaces equipped with base points outside the branch locus. 
If $h:N\to M$ is a branched cover of a 3-manifold $M$ and $L$ is a link in $M$ containing a subset of the branch locus only as its connected components, then} 
$h^{-1}(L)$ denotes the link in $N$ defined by its preimage. 
When $h$ is Galois (i.e., regular), $\Gal(h)= {\rm Deck}(h)$ denotes the Galois group, that is, the group of covering transformations of $N$ over $M$.
\end{nota}

\section{Very admissible links}
In the previous paper \cite{Niibo1}, the notion of  a very admissible knot set (link) $\mca{K}$ in a 3-manifold $M$ was introduced,  
\red{and was regarded} 
as an analogue object of the set of all the primes in a number field.
 However, the construction was not sufficient.  

In this section, we refine the definitions, prove the existence of a very admissible link $\mca{K}$ in \magenta{any closed, oriented, connected} 3-manifold $M$, and give some remarks. 

\violet{Id\`ele theory sums up all the local theories and describes the global theory. In number theory, each prime equips local theory, and in an extension of number field $F/k$, every prime of $F$ is above some prime of $k$.   
In 3-dimensional topology, local theories are the theories of branched covers over the tubular neighborhoods of  knots.} We define a very admissible link as follows, \magenta{and regard it as an analogue of the set of all the primes.} 

\begin{defn}
Let $M$ be a closed, oriented, connected 3-manifold. 
\red{Let} $\mca{K}$ be a link in $M$ \magenta{consisting of} \cyan{countably many (finite or infinite)} \magenta{\emph{tame} components.}  
We say $\mca{K}$ is an \emph{admissible link} of $M$ if the components of $\mca{K}$ generates $H_1(M)$. 
We say $\mca{K}$ is a \emph{very admissible link} of $M$ if for any finite 
cover $h:N\to M$ branched over a finite link in $\mca{K}$, the components of the link $h^{-1}(\mca{K})$ generates $H_{1}(N)$. 
\end{defn}  

\fgreen{Note that for a link consisting of countably infinite disjoint tame knots,  
by the Sielpi\'{n}ski theorem (\cite[Theorem 6.1.27]{EngelikingGT1989}), 
the notion of a \emph{component} makes sense in a natural way, 
that is, each connected component of its image is the image of some $S^1$ in the domain.}  

\magenta{
We may assume that $M$ is a $C^\infty$-manifold. 
We fix a finite $C^\infty$-triangulation $T$ on $M$. 
\fgreen{A knot $K:S^1\to M$} is called \emph{tame} if it satisfies the following equivalent conditions: 
(1) There is a self-homeomorphism $h$ of $M$ such that $h(K)$ is a subcomplex of some refinement of $T$. 
(2) There is a self-homeomorphism $h$ of $M$ such that $h(K)$ is a $C^\infty$-submanifold of $M$. 
(3) There is a tubular neighborhood of $K$, that is, a topological embedding $\iota_K:S^1\times D^2\to M$ with $\iota_K(S^1\times 0)=K$. } 

We note that ($\sharp$) if a neighborhood $V$ of $K$ is given, then $h$ in (1) and (2) can be taken so that 
it has a support in $V$ (i.e., it coincides with {\rm id} on $M-V$).

\begin{proof} 
(1) $\imp$ (2) : We may assume that $K$ itself is a subcomplex of some refinement $T'$ of $T$. 
For each $0$-simplex $v$ of $T'$ on $K$, by a self-homeomorphism of $M$ with support in a small neighborhood of $v$, 
we can modify $K$ so that $K$ is C$^\infty$ in a neighborhood of $v$. 
Doing the similar for every $v$, we obtain (2). 

\noindent 
(2) $\imp$ (3): We may assume that $K$ itself is a C$^\infty$-submanifold of $M$. 
A tubular neighborhood of $V$ is the total space of a $D^2$-bundle on $K\cong S^1$. 
Since $M$ is oriented, $V$ is orientable and hence is the trivial bundle. Hence (3). 

\noindent 
(3) $\imp$ (1): We use \cite[Theorem 5]{Moise1952AffV}: 
Let $M$ be a \red{metrized} 3-manifold with a fix triangulation $T$ and let $K$ be a closed subset of $M$. 
Suppose that there is a neighborhood $V$ of $K$ \red{in $M$} and a topological embedding $\iota:V\to M$ 
so that $\iota(K)$ is a subcomplex of a refinement of $T$. 
Then, there is a self-homeomorphism $h:M\to M$ such that 
$h(K)$ is a subcomplex of a refinement of $T$. 
In addition, for a given $\varepsilon >0$, there is some $h$ with its support in the $\varepsilon$-neighborhood of $K$. Moreover, we can take $h$ as closer to ${\rm id}$ as we want $\cdots$ ($\ast$). 
If we apply this theorem to our $M$ \red{with a metric}, $T,K,V:=\iota_K(S^1\times D^2)$, and the inclusion $\iota$, 
then we obtain (1). 

By noting ($\ast$) and the construction in (1)$\imp$ (2), we see ($\sharp$). 
\end{proof}

We also remark that (1), (2), (3) are equivalent to that $K$ is \emph{locally flat}, i.e., 
for each $x \in K$, there is a closed neighborhood $V$ so that $(V, V\cap K)$ is homeomorphic to $(D^3, D^1)$ as pairs 
(\cite[Theorem 8.1]{Moise1954AffVIII}, \cite[Theorem 9]{Bing1954}). 

A \red{finite} link $L:\sqcup S^1\inj M$ is called \emph{tame} if it satisfies the following equivalent conditions: 
(1) There is a self-homeomorphism $h$ of $M$ such that $h(L)$ is a subcomplex of some refinement of $T$. 
(2) There is a self-homeomorphism $h$ of $M$ such that $h(L)$ is a $C^\infty$-submanifold of $M$. 
(3) Each component $K: S^1\to M$ of $L$ is tame. 

The non-trivial part of this equivalence is to prove that (3) implies (1). We can prove it by (3) $\imp$ (1) for the knot case and the condition ($\sharp$) on the support of a self-homeomorphism $h$. 

\color{black}

\magenta{A finite link consisting of tame components always equips a tubular neighborhood as a link. 
An infinite link $L$ consisting of countably many tame component\red{s} quips a tubular neighborhood as a link 
if and only if it has no accumulation point. We do not eliminate the cases with accumulation points.} 

\magenta{For a tame knot $K$ in $M$, we denote a tubular neighborhood by $V_K$, which is unique up to ambient isotopy. 
For a link $L$ in $M$ consisting of countably many 
tame component\red{s}, we consider \emph{the formal} (or \emph{infinitesimal}) \emph{tubular neighborhood} $V_L:=\sqcup_{K\subset L}V_K$, where $K$ runs through all the components of $L$.} 
\magenta{\emph{The meridian} $\mu_K\in H_1(\del V_L)$ of $K$ is the generator of $\Ker(H_1(\del V_K) \to H_1(V_K))$ corresponding to the orientation of $K$. A \emph{longitude} $\lambda_K\in H_1(\del V_L)$ of $K$ is an element satisfying that $\mu_K$ and $\lambda_K$ form a basis of $H_1(\del V_K)$. We fix a longitude for each $K$. 
For a \magenta{finite} branched cover $h:N\to M$ and for each component of $h^{-1}(K)$ in $N$, we fix a longitude which is a component of the preimage of that of $K$.}

\begin{lem}\label{lemlink} 
Let $M$ be a closed, oriented, connected 3-manifold and let $L$ be a link in $M$ \magenta{consisting of countably many tame components. Then there is a link $\mca{L}$ in $M$ \cyan{containing} $L$, consisting of countably many tame components, and satisfying that}  
for any finite 
cover $h:N\to M$ branched over a finite sublink of $L$, 
$H_1(N)$ is generated by the components of the preimage $h^{-1}(\mca{L})$. \end{lem}

\begin{proof} 
The set of all the finite 
branched covers of $M$ branched over finite sublinks of $L$ is countable, and can be written as  $\(h_i:N_i\to M\)_{i\in \N\red{=\N\cup \{0\}}}$, where $h_0=id_M$. 
\magenta{Indeed, for each finite sublink $L'\subset L$, 
\orange{finite} branched covers of $M$ branched over $L'$ corresponds to subgroups of \orange{$\pi_1(M-L')$} of finite indices. 
\orange{Since $\pi_1(M-L')$} is finitely generated 
group, \orange{such subgroups are countable.}}  


\magenta{We construct an inclusion sequence $L_0\subset L_1\subset \ldots \subset L_i \subset \ldots$ of links consisting of countably many tame components as follows. First, we put $L_{\red 0}=L$. Next, for $i\in \N_{\red{>0}}$, let $L_{i-1}$ be given. We \emph{claim} that there is a link $L_i$ in $M$ including $L_{i-1}$, consisting of countably many tame components, and satisfying that the components of the preimage $h_i^{-1}(L_i)$ generates $H_1(N_i)$. 
By putting $\mca{L}:=\cup_i L_i$, we obtain an expected link.} 


\magenta{
The \emph{claim} above can be  deduced immediately from the following assertion: \emph{For any finite branched cover $h:N\to M$ and the preimage $\wt{L}$ of any link in $M$ consisting of countably many tame components, there is a finite link $L'$ in $N-\wt{L}$  consisting of tame components and the image $h(L')$ being also a link.}} 

\fgreen{Note that $N$ is again a closed, oriented, connected 3-manifold.}
\magenta{
We may assume that $N$ is a $C^\infty$-manifold. 
On the space $C^\infty(S^1,N)$ of maps, since $S^1$ is compact,  
the well-known two topologies called \emph{the compact open topology (the weak topology)} and 
\emph{the Whitney topology (the strong topology)} coincide. 
It is completely metrizable space and satisfies \emph{the Baire property}, that is, 
\emph{for any countable family of open and dense subsets, \red{their} intersection is again dense}. 
(We refer to \cite{Hirsch1994} for the terminologies and the general facts stated here.)}  

\magenta{
Let $\{K_j\}_j$ denote the set of components of $\wt{L}$. 
Since $F_j:=\{K\in C^\infty(S^1,N) \mid K\cap K_j=\phi\}$ is open and dense, by the Baire property, the intersection $F:=\cap_j F_j$ is dense. 
Put $H_1(N)=\langle a_1,\ldots,a_r\rangle$, and 
let $A_1$ denote the set of tame knots $K\in C^\infty(S^1,N)$ satisfying $[K]=a_1$ whose 
image $h(K)$ in $M$ is also a tame knot. 
Then $A_1$ is open and non-empty. Therefore $A_1\cap F$ is non-empty, and we can take an element $K'_1$ of it. For $1\leq k \leq r$, 
if we replace $\wt{L}$ by $\wt{L}\cup K'_1 \cup \ldots \cup K'_k$ and do \red{a similar construction} for $a_{k+1}$ 
successively, then we complete the proof.}  
\end{proof}

\begin{thm}
Let $M$ be a closed, oriented, connected 3-manifold, and $L$ a link in $M$. Then, there is a very admissible link $\mca{K}$ 
\cyan{containing} $L$. 
\end{thm}

\begin{proof}
We construct an inclusion sequence of links $\(\mca{K}_i\)_i$ as follows: 
First, we \red{take a link $\mca{K}_0$ which includes $L$ and generates $H_1(M)$.} 
Next, for $i \in \N_{\red >0}$, let $\mca{K}_{i-1}$ be given, and let $\mca{K}_i$ be a link obtained from $\mca{K}_{i-1}$ by 
\red{Lemma} \ref{lemlink}. Then the union $\mca{K}:=\cup \mca{K}_i$ is a very admissible link. 
\end{proof}

Links $\mca{L}$ and $\mca{K}$ in the lemma and theorem above may be taken smaller than in the constructions. It may be interesting to ask whether they can be finite. %
Let $M=S^3$. The unknot is very admissible link. 
If $L$ is the trefoil, 
by taking branched 2-cover, we see that $\mca{K}_1$ is greater than $L$. 
We expect that $\mca{K}$ has to be infinite. %
Next, let $M$ be a 3-manifold, and $L$ a minimum admissible link ($L$ can be empty). 
For an integral homology 3-sphere $M$, we have $\mca{K}=L=\phi$. 
For a lens space $M=L(p,1)$ or $M=S^2\times S^1$, 
we can take a knot \red{(the core loop)} 
$\mca{K}=L=K$. %

In the latter sections of this paper, we assume that a very admissible link $\mca{K}$ is an infinite link. 
However, our argument are applicable for finite $\mca{K}$ also. 

\begin{rem}[variants]
\noindent (1) \orange{In the definition of a very admissible $\mca{K}$ of $M$, we consider every finite branched covers which are necessarily abelian, so that for each finite abelian branched cover $h:N\to M$, the preimage $h^{-1}(\mca{K})$ is also a very admissible link of $N$.} 
We will discuss a weaker condition on 
\orange{$\mca{K}$} in the end of Section 5, \textbf{Remark} \ref{remonK}. 

\noindent 
(2) 
Let $L$ be an infinite link such that any (ambient isotopy class of) finite link in $M$ is contained in $L$. 
There exists such a link. Indeed, 
since the classes of finite links are countable, 
by putting links side by side in $S^{3}=\R^{3}\cup \(\infty\)$, we obtain such a link $L$, 
with one limit point at $\infty$. 
By using the metric of $\R^{3}$, we can take a tubular neighborhood of $\mca{L}$. 
If we start the construction form such an infinite link, then 
we obtain a special very admissible link $\mca{K}$, which controls all the 
\orange{finite} branched covers of $M$ branched over any finite links in $M$. 
\end{rem}

According to \cite{Morishita2012}, counterparts of infinite primes are ends of 3-manifolds. 
F.\ Hajir also studies cusps of hyperbolic 3-manifolds as analogues of  infinite primes of number fields (\cite{Hajir2012}). 
In this paper, 
since we deal with closed manifolds, the counterpart of the set of infinite primes is empty.

\section{The universal $\mca{K}$-branched cover}
Class field theory deals with all the abelian extension of a number field $k$ in a fixed algebraic closure $\ol{k}$ of $k$. 
In this section, we briefly review the analogies between knots and primes. Then, 
for a 3-manifold $M$ equipped with an infinite (very admissible) link $\mca{K}$, we introduce the notion of the universal $\mca{K}$-branched cover, which is an analogue of an algebraic closure of a number field. We also discuss the role of base points.\\

The analogies between knots and primes has been studied systematically by 
B.\ Mazur (\cite{Mazur1963}), M.\ Kapranov (\cite{Kapranov1995}), A.\ Reznikov (\cite{Reznikov1997}, \cite{Reznikov2000}), M.\ Morishita (\cite{Morishita2002}, \cite{Morishita2010}, \cite{Morishita2012}), 
A.\ Sikora (\cite{Sikora2003}) and others, 
and their research is called arithmetic topology. 
Here is a basic dictionary of the analogies.
For a number field $k$, let $\mca{O}_k$ denote the ring of integer. 
\[
\begin{tabular}{|c||c|} \hline 
3-manifold $M$ & number ring $\Spec \mca{O}_k$ \\ 
knot $K:S^1\inj M$ & prime $\p: \Spec(\F_{\p}) \inj \Spec \mca{O}_k$ \\ 
link $L=\(K_1,...,K_r\)
$ & set of primes $S=\(\p_1,...,\p_r\)$\\ 
(branched) cover $h:N\to M$ & (ramified) extension $F/k$\\ \hline 
fundamental group $\pi_1(M)$ & \'etale fundamental group $\pi_1^{\text{\'et}}(\Spec \mca{O}_k)$\\ 
\magenta{$\pi_1(M-L)$} & \magenta{$\pi_1^{\text{\'et}}(\Spec \mca{O}_k -S)$} \\\hline 
1st homology group $H_1(M)$ 
& ideal class group $\Cl(k)$
\\ \hline 
\end{tabular}
\] 
\red{We put $\ol{\Spec \mca{O}_k}=\Spec \mca{O}_k\cup\{\text{infinite primes}\}$.} 
There is also
an analogy between the Hurewicz isomorphism and the Artin reciprocity in unramified class field theory: 
\[
\begin{tabular}{|c||c|} \hline
 $H_1(M) \cong \Gal(M^{\rm ab}/M) \cong \pi_1(M)^{\rm ab}$& $\Cl(k) \cong \Gal(k^{\rm ab}_{\rm ur}/k)\cong \pi_1^{\text{\'et}}(\red{\ol{\Spec \mca{O}_k}})^{\rm ab}$\\  \hline 
\end{tabular}
\]
Here, $M^{\rm ab}\to M$ and $k^{\rm ab}_{\rm ur}/k$ denote the maximal abelian cover and the maximal unramified abelian extension respectively. 
Moreover, there are branched Galois theories in a parallel manner, where the fundamental groups of the exteriors of knots and primes dominate the branched covers and the ramified extensions respectively. 
Our project of id\`elic class field theory for 3-manifolds aims to pursue the research of analogies in this line. For more analogies, we consult \cite{Ueki1}, \cite{Ueki2}, \cite{Ueki3}, \cite{MTTU}, and \cite{Niibo1} \fgreen{also}.\\ 

 


In the following, we discuss an analogue of an/the algebraic closure of a number field. 
If we say branched covers, 
unless otherwise mentioned, we 
consider \emph{branched covers endowed with base points}, that is, we fix base points in all spaces that are compatible with covering maps. For a space $X$, we denote by $b_X$ the base point.\\ 

First, we recall the notion of an isomorphism of branched covers. 
For covers $h:N\to M$ and $\fgreen{h'}:N'\to M$ branched over $L$, 
we say they are \emph{isomorphic} (as branched covers endowed with base points) and denote by $h\cong h'$ if 
there is a (unique) homeomorphism $f:(N,b_N)\congto (N',b_{N'})$ such that $h=h'\circ f$.  
Let $\ul{h}:Y_L\to X_L$ and $\fgreen{\ul{h}'}:Y'_L\to X_L$ denote the restrictions to the exteriors. 
Then, $h\cong h'$ is equivalent to that $\ul{h}_*(\pi_1(Y_L,b_{Y_L}))=\fgreen{\ul{h}'}_*(\pi_1(Y'_L,b_{Y_{L'}}))$ in $\pi_1(X_L,b_{X_L})$. 

Such notion is extended to the class of branched pro-covers, which are objects obtained as inverse limits of finite branched covers.\\

Next, we introduce an analogue notion of an algebraic closure of a number field. 
For a finite link $L$ in a 3-manifold, a branched pro-cover $h_L:\wt{M_L}\to M$ is a \emph{universal $L$-branched cover} of $M$ if it satisfies a certain universality: 
$h_L:\wt{M_L}\to M$ is a minimal object such that any finite cover of $M$ branched over $L$ factor through it. 
It is unique up to the canonical isomorphisms, and it can be obtained by Fox completion of a universal cover of the exterior $\ul{h_L}:\wt{X_L}\to X_L$. 
(Note that Fox completion is defined for a spread of locally connected T$_1$-spaces in general. (\cite{Fox1957})) 

Now, let $M$ be a 3-manifold equipped with an infinite (very admissible) link $\mca{K}$. A branched pro-cover $h_{\mca{K}}:\wt{M_{\mca{K}}}\to M$ is a \emph{universal $\mca{K}$-branched cover} of $M$ if it satisfies a certain universality: 
$h_{\mca{K}}:\wt{M_{\mca{K}}}\to M$ is a minimal object such that any finite cover of $M$ branched over a finite link $L$ in $\mca{K}$ factor through it. 

It can be obtained as the inverse limit of a family of universal $L$-branched covers, as follows: 
For each finite link $L$ in $\mca{K}$, let $h_L:\wt{M_L}\to M$ be a universal $L$-branched cover of $M$. 
By the universality, for each $L\subset L'$, we have a unique map $f_{L,L'}:\wt{M_{L'}} \to \wt{M_L}$ such that $h_{L'}=h_L\circ f_{L,L'}$. Thus $\{h_L\}_{L\subset \mca{K}}$ forms an inverse system. 
By putting $\wt{M_{\mca{K}}}=\varprojlim_{\violet{L \subset \mca{K}}} \wt{M_L}$, 
we obtained a universal $\mca{K}$-branched cover $h_{\mca{K}}:\wt{M_{\mca{K}}}\to M$ as 
the composite of the natural map $\wt{M_{\mca{K}}}\to \wt{M_L}$ and $h_L$. 

For the universal $\mca{K}$-branched cover, the inverse limit $\pi_1(X_{\mca{K}})$ of the fundamental groups of exteriors $\pi_1(X_L)$ $(L\subset \mca{K})$ acts on it in a natural way. The finite branched covers of $M$ obtained as quotients of $h_{\mca{K}}$ by subgroups of $\pi_1(X_{\mca{K}})$ form a complete system of representatives of the isomorphism classes of covers of $M$ branched over links in $\mca{K}$.

Therefore, in the latter section of this paper, if we take $(M,\mca{K})$, 
we silently fix a universal $\mca{K}$-branched cover, 
call it ``the'' universal $\mca{K}$-branched cover,
and restrict our argument to the branched subcovers obtained as its quotients.\\

Finally, we discuss an analogue of a base point.
The following facts explain the role of base points in branched covers: 

\begin{prop}
\noindent 
(1) For $(M,\mca{K})$, we fix a universal $\mca{K}$-branched cover $h_{\mca{K}}$. Then, for a branched cover $h:N\to M$ whose base point is forgotten, 
taking a branched pro-cover $f:\wt{M_{\mca{K}}}\to N$ such that $h\circ f=h_{\mca{K}}$ is equivalent to fixing a base point in $N$ such that $h(b_N)=b_M$. 

\noindent 
(2) Let $h:N\to M$ be a branched cover. Then, a base point of a universal $\mca{K}$-branched cover $h_{\mca{K}}$ 
defines a branched pro-cover $f:\wt{M_{\mca{K}}}\to N$ such that $h_{\mca{K}}=h\circ f$. 
\end{prop}

An analogue of a base point in a 3-manifold is a geometric point of a number field. 
Let $\Omega$ be a sufficiently large field which includes $\Q$, for instance, $\Omega=\C$. 
Then, for a number field $k$, choosing a geometric point $x:\Spec \Omega \to \Spec \mca{O}_k$ is equivalent to choosing an inclusion $k\inj \Omega$.  
Moreover, choosing base points in a cover $h:N\to M$ which are compatible with the covering map is an analogue of choosing inclusion $k\subset F\inj \Omega$ for an extension $F/k$. 
For an algebraic closure $\ol{k}/k$ and an extension $F/k$ of a number field $k$, we have following facts: 

\begin{prop}
\noindent 
(1) If we fix $\ol{k}/k$ in $\Omega$, 
taking an inclusion $F\inj \ol{k}$
is equivalent to taking an inclusion $F\inj \Omega$. 

\noindent 
(2) For an extension $F/k$ in $\Omega$, 
an inclusion $\ol{k}\inj \Omega$ defines $F\inj \ol{k}$. 
\end{prop} 

In addition, we have 
$\Spec \mca{O}_k=\{{\rm finite\ primes}\}\cup \Spec k$, and 
$(\Spec k)(\Omega)=\{\Omega{\rm \mathchar`-rational\ points\ of} \Spec k\}
:=\Hom (\Spec \Omega, \Spec k)
\cong \Hom (k,\Omega)$.  
Accordingly, choosing a geometric point (an injection) $k\inj \Omega$ is an analogue of choosing a base point in the exterior of $\mca{K}$ in $M$. 
If $k/\Q$ is Galois, we have a non canonical isomorphism
$\{$the choices of a geometric point of $k\}=\Hom (k,\Omega) \cong \Gal (k/\Q)$. 
This map depends on the fact that an inclusion of $\Q$ into a field is unique. 
In order to state an analogue for $(M,\mca{K})$, we need to fix an analogue of $k/\Q$. 
If we fix a Galois branched cover $h_M:M\to S^3$ whose base point is forgotten, an infinite link $\ul{\mca{K}}$ in $S^3$ such that $h^{-1}(\ul{\mca{K}})=\mca{K}$, and a base point $b_0$ in $S^3$, 
then we have a non-canonical map $\{$the choices of base points in $M \}\cong \Gal(h_M)$. \\

Thereby, we obtained the following dictionary: 
\[
\begin{tabular}{|c||c|} \hline
3-manifold with very admissible link 
$(M,\mca{K})$& 
number ring 
$\Spec \mca{O}_k$\\
\hline 
universal $\mca{K}$-branched cover $h_\mca{K}:\wt{M_\mca{K}}\to M$&algebraic closure $\ol{k}/k$\\
 \hline 
 base point $b_M: \{{\rm pt}\} \inj M$
& geometric point $x:\Spec \Omega \to \Spec \mca{O}_k$\\  \hline 
\end{tabular}
\] 

In this paper, since we consider only regular (Galois) covers, we can forget  base points. Then weaker equivalence classes of branched covers should be considered.

\section{Id\`elic class field theory for number fields}

In this section, we briefly review the id\`elic class field theory for number fields, whose topological analogues will be studied in later sections. 
We consult \cite{KKS2} and \cite{Neukirch} as basic references for this section. 

\subsection{Local theory}

We firstly review the local theory. 
Let $k$ be a number field, that is, a finite extension of the rationals $\Q$,
and let $\p\subset \mca{O}_k$ be \red{a prime ideal of its integer ring}. 
Then, for a local field $k_\p$, we have the following commutative diagram of splitting exact sequences.
$$
\xymatrix{
\red{1}\ar[r] &\mca{O}_\p^\times \ar[r] \ar[d] & k_\p^\times \ar[r] ^{v_\p} \ar[d]^{\rho_\p}&\Z \ar[r] \ar[d] &0\\%
\red{1}\ar[r] &\Gal(k_\p^{\rm ab}/k_\p^{\rm ur})\ar[r] &\Gal(k_\p^{\rm ab}/k_\p) \ar[r] &\Gal(k_\p^{\rm ur}/k_\p) \ar[r] &\red{1}}$$
Here, $\mca{O}_\p^{\times}$ is the local unit group, $v_\p$ is the \emph{valuation}, 
$k_\p^{\rm ab}/k_\p$ is the maximal abelian extension, and $k_\p^{\rm ur}/k_\p$ is  the maximal unramified abelian extension. The map $\rho_\p$ is called \emph{the local reciprocity homomorphism}, which is a canonical injective homomorphism with dense image, and controls all the abelian extensions of the local field $k_\p$. 
In the lower line,  $I_\p^{\rm ab}=\Gal(k_\p^{\rm ab}/k_\p^{\rm ur})$ is the abelian quotient of the inertia group, and we have $\Gal(k_\p^{\rm ur}/k_\p) \cong \Gal(\ol{\F}_\p/\F_\p)\cong \wh{\Z}$. 

The theory of a local field is rather complicated. 
There are non-abelian extensions, and there are notions of wild and tame for ramifications. 
For the tame quotients, we have an exact sequence 
\red{
$$1\to I_\p^t\to \Gal(\ol{k}_\p/k)\to \Gal(\ol{\F}_\p/\F_\p)\to 1,$$ 
where $I_\p^t=\langle\tau\rangle\cong \prod_{l\neq p}\Z_l$, 
$\Gal(\ol{k}_\p	/k)= \pi_1^t(\Spec (k_\p))=\langle \tau,\sigma \mid \tau^{q-1}[\tau,\sigma]\rangle$, 
and $\Gal(\ol{\F}_\p/\F_\p)=\langle \sigma \rangle\cong \wh{\Z}\cong \prod_p \Z_p$.
They are topologically generated by a monodromy $\tau$ and the Frobenius $\sigma$. }






\red{The local theory of an infinite prime $\p:k \overset{\wt{\p}}{\inj} \C\to \R_{\geq0}; x\mapsto |\wt{\p}(x)|$ is described as follows. If $\p$ is real, then $v_\p:k^\times \to \R; x\mapsto \log|\wt{\p}(x)|$ yields an exact sequence $1\to \{\pm 1\}\to \R^\times \overset{v_\p}{\to}\R\to 0$. By taking Housdorffication with respect to the local norm topology, 
we obtain an exact sequence $1\to \{\pm 1\}\to \{\pm1\}\to 0\to 0$. 
If $\p$ is complex, then we have an exact sequence $1\to S^1\to \C^\times \overset{v_\p}{\to}\R\to 0$, and obtain an exact sequence $1\to 1\to 1\to 0 \to 0$ of trivial terms in a similar way. 
We put $\mca{O}_\p^\times=\{\pm 1\}$ or $1$ according as $\p$ is real or complex.  
In both cases, there are commutative diagrams similar to the case of finite primes. }  



\subsection{Definitions} 
Next, we review the global theory. 
Let $k$ be a number field. We define the {\it id\`ele group} $I_k$ of $k$ by the following restricted product of $k_{\p}^\times$ with respect to the local unit groups 
\red{$\mca{O}_\p^\times$}  
 over all \red{finite and infinite} primes ${\p}$ of $k$:
\[I_k:=\restprod_{\p}k_\p^{\times}=
\(\, (a_\p)_{\p} \in \prod_{\p:\text{ prime}} k_\p^{\times} \ \middle|\ 
 v_{\p}(a_{\p})=0 \text{ for almost all finite primes } \p\).\] 
\red{This is the restricted products with respect to the local topology on $k_\p^{\times}$ (see Subsection \ref{st topology}) and the family of open subgroups $\(\mca{O}_\p^{\times}< k_\p^{\times}\)_\p$.}  

Since we have $v_{\p}(a)=0$ for $a \in k^\times$ and for almost all finite primes $\p$, 
$k^\times$ is embedded into $I_k$ diagonally. 
We define the {\it principal id\`ele group} $P_k$ of $k$ by the image of 
\magenta{the diagonal embedding $\Delta:k^\times \to I_k$}, 
and the {\it id\`ele class group} of $k$ by the quotient 
$C_k := I_k/P_k .$ 

Then, the homomorphism to the ideal group 
$\varphi:I_k \to I(k)
 ; \  
(a_\p)_\p \mapsto \prod_\p \p^{v_{\p}(a_\p)} $ 
induces an isomorphism 
\[
I_k/(U_k\cdot P_k) \cong \Cl(k),\]
where $U_k=\Ker \varphi=\prod_\p \red{\mca{O}_\p^\times}$ 
denotes the unit id\`ele group and $\Cl(k)$ denotes the ideal class group of $k$. 

\subsection{Standard topology} \label{st topology}

The id\`ele class group $C_k$ equips the \emph{standard topology}, which is  the quotient topology of the \emph{restricted product topology} on the id\`ele group $I_k$ of the local topologies, defined as follows.

We firstly consider on $\mca{O}_\p^{\times}$ the relative topology of \emph{\red{the} local norm topology} of $k_\p^{\times}$, and re-define the \emph{local topology} on $k_\p^{\times}$  as the unique topology such that the inclusion $\mca{O}_\p^{\times}\inj k_\p^{\times}$ is open and continuous. 
Next, for each finite set of primes $T$ which includes all the infinite primes, we consider the product topogy on 
$G(T)=\prod_{\p \in T}k_\p^{\times}\times \prod_{\p\not\in T}\mca{O}_\p^{\times}$. 
Then, we define the \emph{standard topology} on $I_k$ so that 
each subgroup $H<I_k$ is open if and only if $H\cap G(T)$ is open for every $T$. 

This standard topology on $C_k$ differs from the one defined as the quotient topology of relative topology of product topology of the local topologies on $I_k< \prod k_\p^{\times}$, and it is finer than the latter. 


\subsection{Norm topology}

For a finite abelian extension $F/k$,  
the norm map $N_{F/k}:C_F \to C_k$ is defined as follows. 

Let $\p$ be a prime of $k$ and $F_{\p}^\times:= \prod_{\P|\p} F_\P^\times$.  
Each $\alpha_\p\in F_{\p}^\times$ defines a $k_\p$-linear automorphism $\alpha_{\p}:F_\p^\times \to F_\p^\times;\ x\mapsto \alpha_{\p}x$,
and the norm of $\alpha_{\p}$ is defined by
$N_{F_{\p}/k_{\p}}(\alpha_{\p})= \det(\alpha_{\p}).$ 
It induces a homomorphism
$N_{F_{\p}/k_{\p}}:F_{\p}^\times \to k_{\p}^\times,$ 
and the norm homomorphism 
$N_{F/k}: I_F \to I_k$ on the id\`ele groups. 
Since $N_{F/k}$ sends the principal id\`eles to principal id\`eles, 
it also induces the norm homomorphism $N_{F/k}:C_F\to C_k$ on the id\`ele class groups.

For a number field $k$, the id\`ele class group $C_k$ equips the {\it norm topology}, so that it is a topological group, and  the family of $N_{F/k}(C_F)$ is a fundamental system of neighborhoods of $0$, 
where $F/k$ runs through all the finite abelian extensions of $k$. 
\begin{prop}A subgroup $H$ of $C_k$ is open and of finite index with respect to the standard topology if and only if it is open with respect to the norm topology.\end{prop} 

\subsection{Global theory}
Here is a main theorem of id\`elic class field theory for number fields (cf. \cite{Neukirch}, \S 6, Theorem 6.1): 

\begin{thm}[Id\`elic class field theory for number fields]\label{global CFT} Let $k$ be a number field and let $k^{\rm ab}$ denote the maximal abelian extension of $k$ which are fixed in $\C$.\ \\
{\rm (1) (Artin's global reciprocity law.)} There is a canonical surjective homomorphism, called the global reciprocity map,
\[\rho_{k}:C_k \to \Gal(k^{\rm ab}/k)\]
which has the following properties:\\
{\rm (i)} For any finite abelian extension $F/k$ in $\C$,
$\rho_k$ induces an isomorphism
\[C_k/N_{F/k}(C_F)\cong \Gal(F/k).\]
{\rm (ii)} For each prime $\p$ of $k$, we have the following commutative diagram
\[\xymatrix{
k_\p^{\times} \ar[d]_{\iota_\p} \ar[r]^(0.3){\rho_{k_\p}} \ar@{}[dr]|\circlearrowleft & \Gal(k_\p^{\rm ab}/k_{\p}) \ar[d]^{} \\
C_k \ar[r]_(0.3){\rho_k} & \Gal(k^{\rm ab}/k), 
}\]
where $\iota_{\p}$ is the map induced by the natural inclusion $k_\p^\times \to I_k$. 
\ \\
{\rm (2) (The existence theorem.)} The correspondence 
$$F\mapsto \mca{N}=N_{F/k}(C_F)$$
gives a bijection between the set of finite abelian extensions $F/k$ in $\C$ and
the set of open subgroups $\mca{N}$ of finite indices of $C_k$ with respect to the standard topology. 
Moreover, the latter set coincides with 
the set of open subgroups of $C_k$ with respect to the norm topology.
\end{thm} 

In the proof, we use the {\it norm residue symbol} $(\ ,F/k): C_k\surj \Gal(F/k)$. 
For this map, we have $\Ker (\ ,F/k)=N_{F/k}(C_F)$.


\section{The global reciprocity law}

In this section, we recall the local theory, 
\orange{develop} the id\`elic class field theory for 3-manifolds 
\orange{with a new definition of the principal id\`ele group}, 
and present the global reciprocity law over a 3-manifold equipped with a 
link. 
\orange{We generalize the main result of the previous paper \cite{Niibo1}, as well as 
suggest an expansion of the M$^2$KR-dictionary.}

\subsection{Local theory}
Let $K$ be a knot in its tubular neighborhood $V_K$. 
In our context, the local theory for 3-manifolds is nothing but the Galois theory for the covers of $\del V_K$, 
which is dominated by an abelian group $\pi_1(\del V_K)=\langle \mu_K,\lambda_K \mid [\mu_K,\lambda_K] \rangle \cong H_1(\del V_K) \cong \Z^2$. 
(In a sense, the tame case of a local field is a ``quantized'' version of this case.) 
For each manifold $X$, let $\Gal(\wt{X}/X)$ denote the Galois group of the universal cover. 
We have the following commutative diagram of exact sequences. 
$$
\xymatrix{
0\ar[r] &\langle \mu_K\rangle \ar[r] \ar[d] & H_1(\del V_K) \ar[r] ^{v_K} \ar[d]^{{\rm Hur.}}&H_1(V_K)=\langle\lambda_K\rangle \ar[r] \ar[d] &0\\%
0\ar[r] &\Gal(\wt{\del V_K}/\del\wt{V_K})\ar[r] &\Gal(\wt{\del V_K}/\del V_K) \ar[r] &\Gal(\wt{V_K}/V_K) \ar[r] &0}$$
By an isomorphism $\del_*: H_2(V_K,\del V_K)=\langle[D_K]\rangle \congto \langle\mu_K\rangle$, classes of meridian disks $D_K$ can be seen as analogues of local units. \red{(An analogue of the unit group of a number field is considered as a surface-related object, while there are several variations of dictionary about units.)} 
The map $v_K$ is the one induced by the natural injection $\del V_K\inj V_K$, which is an analogue of the valuation map $v_\p$ in number theory. 
The vertical maps are \red{(the inverses of)} the Hurewicz isomorphisms. In the lower line, $I_K:=\Gal(\wt{\del V_K}/\del\wt{V_K})$ is the inertia group, and we have $\Gal(\wt{V_K}/V_K) \cong \Gal(\wt{K}/K)\cong \Z$. 
\subsection{Definitions}
Let $M$ be a closed, oriented, connected 3-manifold. 
Let $\mca{K}$ be a link in $M$ 
\magenta{consisting of countably many tame components with a (the) formal} tubular neighborhood $V_\mca{K}=\sqcup_{K\subset \mca{K}} V_K$.  For a sublink $L$ of $\mca{K}$,  
\magenta{we put $V_L=\sqcup_{K\subset L} V_K$}. 

\begin{defn}[id\`ele group] 
We define the {\it id\`ele group} of $(M,\mca{K})$ by the 
restricted product of $H_1(\del V_K)$ with respect to the subgroups $\(\langle \mu_K\rangle\)_{K\subset \mca{K}}=\(\Ker(v_K)\)_{K\subset \mca{K}}$: 
\[ \ds I_{M,\mca{K}}:=
\restprod_{K\subset \mca{K}}H_1(\del V_K)=
\((a_K)_K\in \prod_{K\subset \mca{K}}H_1(\del V_K)\ \middle| \ 
v_{K}(a_K)=0\\ {\rm\ for\ almost\ all\ }K\).\] 
\end{defn}
\red{
This is the restricted product with respect to the local topology on $H_1(\del V_K)$ (see Section \ref{standard}) and the family of open subgroups $\{\langle\mu_K\rangle <H_1(\del V_K)\}_K$.}\\ 




\cyan{
The set of finite sublinks of $\mca{K}$ is a directed ordered set with respect to the inclusions. 
In addition, if we take an inclusion sequence $\cdots \subsetneq  L_i\subsetneq L_{i+1}\subsetneq \cdots$ of finite sublinks of $\mca{K}$ indexed by $i\in \N$, then $\mca{K}=\cup_i L_i$ and any finite sublink $L$ of $\mca{K}$ is contained in some $L_i$.
} 

\cyan{
For each finite sublink $L$ of $\mca{K}$, we put $X_L=M-L$. Then $H_1(X_L)$'s form an inverse system indexed by $L\subset \mca{K}$ with respect to the natural surjections induced by the inclusion maps of the exteriors. By $\varprojlim_L H_1(X_L)\cong \varprojlim_{i\in \N} H_1(X_{L_i})$, the natural projection $\varprojlim_L H_1(X_L)\to H_1(X_L)$ is surjective. 
}

\cyan{
If we put $X_\mca{K}=M-\mca{K}$, then we have $H_1(X_\mca{K})\cong \varprojlim_L H_1(X_L)$. 
Indeed, there is the Milnor exact sequence $0\to \varprojlim_L^1 H_2(X_L)\to H_1(X_\mca{K})\to \varprojlim_L H_1(X_L)\to 0$ (\cite{Milnor1962ax}). Since $\(H_2(X_L)\)_L$ is a surjective system and satisfies the Mittag-Leffler condition, we have $\varprojlim_L^1 H_2(X_L)=0$.
}  

\cyan{
For each $L$, 
let $\Gal(X_{L}^{\rm ab}/X_{L})$ denote the Galois group of the maximal abelian cover over its exterior $X_L$. Then $\Gal(X_{L}^{\rm ab}/X_{L})$'s form an inverse system in a natural way. 
We put 
$\Gal (M,\mca{K})^{\rm ab}:=\varprojlim_L \Gal(X_{L}^{\rm ab}/X_{L})$ and regard it as an analogue of $\Gal(k_{\rm ab}/k)$. 
We have $\ds \Gal (M,\mca{K})^{\rm ab} \underset{\rm Hur}{\cong} H_1(X_\mca{K})$. 
}


For a knot $K$ and a finite link $L$ in $\mca{K}$, 
\magenta{take an ambient isotopy $h$ fixing $K$ and $L$ so that $h(V_K)\subset X_L$ if needed. 
Then the composite $\del V_K\congto h(\del V_K)\inj X_L$ with the inclusion induces a natural map $H_1(\del V_K)\to H_1(X_L)$ commuting with the the Hurewicz maps.} 
\[
\xymatrix{
H_1(\del V_K) \ar[d] \ar[r]^{\cong \ \ \ \ }_{\rm Hur \ \ \ \ } \ar@{}[dr]|\circlearrowleft 
& \Gal(\wt{\del V_K}/\del V_K) \ar[d] \\
H_1(X_L) \ar[r]^{\cong \ \ \ \ }_{\rm Hur \ \ \ \ }  & \Gal(X_{L}^{\rm ab}/X_L)}
\] 
Let $\rho_{K,L}: H_1(\del V_K)\to \Gal(X_L^{\rm ab}/X_L)$ denote their composite, 
and \magenta{we consider the map} $\ds \rho_L: I_{M,\mca{K}}\to \Gal(X_{L}^{\rm ab}/X_L): (a_K)_K\mapsto \sum_{K\subset\mca{K}} \rho_{K,L}(a_K)$ \magenta{where $K$ runs through all the knot in $\mca{K}$.} 
This sum makes sense, because it is actually a finite sum for each $(a_K)_K\in I_{M,\mca{K}}$, by the definition of the restricted product.
Since $(\rho_L)_L$ is compatible with the inverse system, the following homomorphism is induced: 
$$\cyan{\wt{\rho}_{M,\mca{K}}}:I_{M,\mca{K}}\to \Gal(M,\mca{K})^{\rm ab}.$$

If $\mca{K}$ is an admissible link, \magenta{then} this map is surjective.\\ 


We give a new definition of the principal id\`ele group and id\`ele class group \fgreen{by introducing the natural homomorphism $\Delta: H_2(M,\mca{K})\to I_{M,\mca{K}}$} in the following. 

Let $L$ and $L'$ be finite sublinks of $\mca{K}$ with $L\subset L'$. 
Then the natural surjection $j:C_*(M,L)\surj C_*(M,L')$ induces the natural injection $j_*:H_2(M,L)\inj H_2(M,L')$. 
\fgreen{We have the natural isomorphism $H_2(M,\mca{K})\congto \varinjlim_{L\subset \mca{K}}H_2(M,L)$, }
where $L$ runs through all the finite sublinks of $\mca{K}$ and the transition maps are the natural map $j_*$'s. 
\fgreen{
In order to prove this, we use the Sielpi\'{n}ski theorem (\cite[Theorem 6.1.27]{EngelikingGT1989}): 
If a compact Hausdorff connected space $X$ and a countable family $\{X_i\}_{i\in \N}$ of pairwise disjoint closed subsets satisfy $X=\cup_i X_i$, then at most one of $X_i$ is non-empty. 
By virtue of this theorem, the singular chain groups satisfy $C_n(\mca{K})=\varinjlim_{L\subset \mca{K}} C_n(L)$ for each $n\in \N$. 
The exact sequence $0\to C_n(L) \to C_n(M) \to C_n(M,L)\to 0$ yields 
the exact sequence $0\to C_n(\mca{K}) \to C_n(M) \to \varinjlim_{L\subset \mca{K}} C_n(M,L)\to 0$. 
The exact sequence $0\to C_n(\mca{K}) \to C_n(M) \to C_n(M,\mca{K})\to 0$ induces the natural isomorphism 
$C_n(M,\mca{K})\congto \varinjlim_{L\subset \mca{K}} C_n(M,L)$. 
By taking the long exact sequences and using the five lemma, we obtain the natural isomorphism $H_n(M,\mca{K})\congto \varinjlim_{L\subset \mca{K}} H_n(M,L)$.}

\cyan{
For each finite sublink $L$ of $\mca{K}$, let $V'_{L}$ be a (usual) tubular neighborhood of $L$ and put $X_{L}^\circ=M-{\rm Int}(V'_{L})$. 
The inclusions $(M,L)\inj (M,V'_L)$ and $(X_L^\circ,\del X_L^\circ)\inj (M,V'_L)$ induce isomorphisms $H_2(M,L) \cong H_2(M,V'_L) \cong H_2(X_L^\circ,\del X_L^\circ)$.
We denote by $\del_L$ the homomorphism $H_2(M,L)\to H_1(\del V_L)$ given as the composite of
$\del_*: H_2(M,L) \cong H_2(X_L^\circ,\del X_L^\circ)\to H_1(\del X_L^\circ)$ and a natural
isomorphism $H_1(\del X_L^\circ)=H_1(\del V'_L)\congto H_1(\del V_L)$. 
We also consider the homomorphism $H_1(\del V_L)\congto H_1(\del V'_L) \to H_1(X_L)$.} 
%
For each finite sublinks $L$ and $L'$ of $\mca{K}$ with $L\subset L'$, there is a commutative diagram 
{\small 
$$\xymatrix{
H_2(M,L') \ar@{->>}[r]^{\del_{L'}} & H_1(\del V_{L'}) \ar@{->>}^{{\rm pr}}[d]\\
H_2(M,L)  \ar@{->>}[r]^{\del_L} \ar@{^{(}->}[u]_{j_*} \ar@{}[ru]|{\circlearrowright}
&H_1(\del V_L)}$$}
%
%
%
where ${\rm pr}$ denotes the projection to the $L$-components. 
Thus a natural map from 
$\varinjlim_{L\subset \mca{K}} H_2(M,L)$ to $\varprojlim_{L\subset \mca{K}}H_1(\del V_L)=\prod_{K\subset \mca{K}} H_1(\del V_K)$ 
is induced. 
Since longitudinal component does not added by $j_*$, the image of this map is included in $I_{M,\mca{K}}$. 
\fgreen{Thus we obtain the natural homomorphism $\Delta:H_2(M,\mca{K})\to I_{M,\mca{K}}$. } 
\magenta{If $M$ is a $\Q$HS$^3$, then $\del_L$ is injective for each finite sublink $L$ of $\mca{K}$, and hence so is $\Delta$.}

\begin{defn}[Principal id\`ele group, id\`ele class group] 
We define \emph{the principal id\`ele group} by $P_{M,\mca{K}}:=\Im (\Delta:\fgreen{H_2(M,\mca{K})}\to I_{M,\mca{K}})$, and \emph{the id\`ele class group} by $C_{M,\mca{K}}:=I_{M,\mca{K}}/P_{M,\mca{K}}$. 
\end{defn} 




\orange{The following assertion tells that our $P_{M,\mca{K}}$ coincides with that of \cite{Niibo1}, 
and implies the existence of the global reciprocity map $\rho_{M,\mca{K}}$ in Theorem \ref{global CFT for mfd}. 
It gives a topological interpretation of $\Ker \wt{\rho}_{M,\mca{K}}$ and 
guarantees the correctness 
of an analogue of the Artin's reciprocity law. }

\begin{thm} \label{Pidele}
\violet{The equation $P_{M,\mca{K}}=\Ker \wt{\rho}_{M,\mca{K}} $ holds, and $\wt{\rho}_{M,\mca{K}}$ induces a} natural isomorphism 
$\rho_{M,\mca{K}}:C_{M,\mca{K}}\congto \Gal(M,\mca{K})^{\rm ab}$. 
\end{thm}

\begin{proof} 
The assertion $\Im \Delta \subset \magenta{\Ker \wt{\rho}_{M,\mca{K}}}$ holds in a natural way. 
Indeed, for any $x\in S_{M,\mca{K}}$, there is some $L_0\subset \mca{K}$ and some $x_0 \in H_2(M,L_0)$ such that $x$ is the image of $x_0$ under the natural map $j:H_2(M,L_0)\inj \fgreen{H_2(M,\mca{K})}$. 
For any finite link $L$ with $L_0\subset L\subset \mca{K}$, there is a commutative diagram 
{\small 
$$\xymatrix{
\fgreen{H_2(M,\mca{K})} \ar[r]^{\Delta} & I_{M,\mca{K}} \ar[r]^{\magenta{\wt{\rho}_{M,\mca{K}}}} \ar@{->>}[d] & \magenta{\Gal(M,\mca{K})^{\rm ab}} \ar[r]  \ar@{->>}[d] & 0\\ 
H_2(\magenta{M,L}) \ar@{^{(}->}[u]_j \ar[r]^{\del_{L}} & H_1(\del \magenta{V_{L}}) \ar[r] & H_1(X_{L}) \ar[r] & 0
}$$}
and the image of $x_0$ in $H_1(X_L)$ is zero. Thus the image of $x$ in $\Gal(M,\mca{K})^{\rm ab}$ is zero, and $\Delta(x)\in \Ker \cyan{\wt{\rho}_{M,\mca{K}}}$ holds. 

We prove $\Ker \cyan{\wt{\rho}_{M,\mca{K}}}\subset \Im \Delta$ in the following. Let $(a_K)\in \Ker \cyan{\wt{\rho}_{M,\mca{K}}}$. Then there is a finite sublink $L\subset \cyan{\mca{K}}$ such that the longitudinal component of $(a_K)$ is zero outside $L$ and that components of $L$ generates $H_1(M)$. 
Let $a$ denote the image of $(a_K)$ in $H_1(\del V_L)$. 
\cyan{The image of $a$ in $H_1(X_L)$ coincides \orange{with} that of $(a_K)$ and hence it is zero. By}
the exact sequence $H_2(\magenta{M,L})\to H_1(\del \magenta{V_{L}})\to H_1(X_{L})\to 0$, 
there is some $A\in H_2(M,L)$ with $\del A=a$. We put $(a'_K)=\Delta(j(A))$. Then it is sufficient to prove $(a_K)=(a'_K)$. 

Let $L'$ be any finite link with $L\subset L'\subset \mca{K}$, and let $b$ and $b'$ denote the images of $(a_K)$ and $(a'_K)$ in $H_1(\del \magenta{V_{L'}})$ respectively. Then it is sufficient to prove $b=b'$. 
Note that $b'$ is the image of $A$ under $H_2(\magenta{M,L})\overset{j_*}{\to} H_2(\magenta{M,L'})\overset{\del_{L'}}{\to}H_1(\del \magenta{V_{L'}})$. 
Now $b$ and $b'$ are both included in $H_1(\magenta{\del V_L})\oplus \fgreen{\langle\mu_K\rangle _{K\subset L'-L}}$, their images in $H_1(X_{L'})$ are zero, and their images in $H_1(\del V_L)$ are $a$.

{\small 
$$\xymatrix{
H_2(M,L') \ar@{->>}[r]^{\del_{L'}} & H_1(\del V_{L'}) \ar[r] \ar@{->>}^{{\rm pr}}[d]  & H_1(X_{L'}) \ar@{->>}^{\iota_*}[d]\\
H_2(M,L)  \ar@{->>}[r]^{\del_L} \ar@{^{(}->}_{j_*}[u] 
& H_1(\del V_L) \ar[r] & H_1(X_L)}$$}

We put $c=b'-b$. Then we have $c\in \langle\mu_{L'-L}\rangle$. 
\cyan{We regard $Z_2(M,L)$ and $Z_2(M,L')$ as subgroups of $C_2(M)$ with $Z_2(M)\subset Z_2(M,L)\subset Z_2(M,L')\subset C_2(M)$, and denote by $\del$ the boundary map on $C_*(M)$.}  
Since the image of $c$ in $H_1(X_{L'})$ is zero, there is some $C\in Z_2(M,L')$ with $\del_{L'}([C])=c$. 

\cyan{Let $V'_{L'}$ be a (usual) tubular neighborhood of $L'$. Then $\del_*:H_2(M,L')\to H_1(L')$ factors as $H_2(M,L')\overset{\del_{L'}}{\to} H_1(\del V_{L'})\congto H_1(\del V'_{L'})\surj H_1(V'_{L'})\congto H_1(L')$ with $\fgreen{\langle\mu_K\rangle _{K\subset L'}}=\Ker (H_1(\del V_{L'})\to H_1(L'))$. 
Since $\del_{L'}([C]) \in \fgreen{\langle\mu_K\rangle _{K\subset L'-L}}$, we have $\del_*[C]=0$, and we can regard $C\in Z_2(M)$. }  

Let $I:H_2(M)\times H_1(M)\to \Z$ denote \emph{the intersection form} of $M$. 
It is a bilinear form 
defined by counting the intersection points of transversely intersecting representatives with signs. 
By the universal coefficient theorem, $H_2(M)$ is torsion-free, and $I$ is right-non-degenerate.  

Now $H_1(M)$ is generated by components of $L$ by assumption. 
Since $\del_L'([C])\in \mu_{L'-L}$, we have $\del_L([C])=0$ by regarding $C\in Z_2(M,L)$, 
and each component $K_i$ of  $L$ satisfies $I([C],[K_i])=0$. 
This implies $[C]=0$ and hence $c=\del_{L'}([C])=0$. 
Therefore we have $b=b'$, and $\Delta:\fgreen{H_2(M,\mca{K})}\to \magenta{\Ker \cyan{\wt{\rho}_{M,\mca{K}}}}$ is a surjection. 
\end{proof}

\orange{
Theorem \ref{Pidele} expands the M$^2$KR-dictionary as follows,} 
where 
$k$ is a number field. 



\begin{center} \begin{tabular}{|c||c|}
\hline 
1-cycle group $Z_1(M)$& ideal group $I(k)$\\ 
\hline 
 $\del: C_2(M)\to Z_1(M); s\mapsto \del s$ & $(\ ): k^\times\to I(k); a\mapsto (a)$ \\ 
\hline 
1-boundary group $B_1(M):=\Im \del $& principal ideal group $P(k):=\Im \del $\\
\hline 
\magenta{$H_1(M):= Z_1(M)/B_1(M)$} & \magenta{$\Cl(k):=I(k)/P(k)$} \\ 
\hline 
\hline 
\orange{id\`ele group} $I_{M,\mca{K}}$ & id\`ele group $I_k$\\ 
\hline 
$\Delta: \fgreen{H_2(M,\mca{K})}\to I_{M,\mca{K}}$ & $\Delta:k^\times \to I_k$ \\ 
\hline 
principal id\`ele group $P_{M,\mca{K}}:=\Im \Delta$ & 
principal id\`ele group $P_k:=\Im \Delta$\\ 
\hline 
id\`ele class group $C_{M,\mca{K}}:=I_{M,\mca{K}}/P_{M,\mca{K}}$ & 
id\`ele class group $C_k:=I_k/P_k$ \\
\hline 
\end{tabular} \end{center}  

\ 



\orange{
Let} $h:N\to M$ be a finite 
branched cover branched over a finite link $L$ in $\mca{K}$. 
Then the preimage $h^{-1}(\mca{K})$ of $\mca{K}$ is a link in $N$, and the covering map $h$ induces 
the {\it norm maps} 
$h_{*}:I_{N,h^{-1}(\mca{K})}\to I_{M,\mca{K}}$,
$h_{*}:P_{N,h^{-1}(\mca{K})}\to P_{M,\mca{K}}$, and
$h_{*}:C_{N,h^{-1}(\mca{K})}\to C_{M,\mca{K}}$. 
They satisfy the transitivity (functoriality) in a natural way.  
\orange{If $\mca{K}$ is very admissible, then so is $h^{-1}(\mca{K})$.}  

\subsection{The global reciprocity law}
Here is the first part of \emph{the id\`elic global class field theory for 3-manifolds and admissible links}, which is the counter part of \textbf{Theorem 3.1} (1). 

\begin{thm}[The global reciprocity law for 3-manifolds]
\label{global CFT for mfd} 
Let M be a closed, oriented, connected 3-manifold  
equipped with a very admissible link $\mca{K}$. 
Then, there is a canonical isomorphism called the \emph{global reciprocity map} 
$$\rho_{M,\mca{K}}:C_{M,\mca{K}} \congto \Gal(M,\mca{K})^{\rm ab}$$ which satisfies the following properties: 

\noindent {\rm (i)} For any finite abelian cover $h:N \to M$ branched over a finite link $L$ in $\mca{K}$, $\rho_M$ induces an isomorphism
$$C_{M,\mca{K}}/h_{*}(C_{N,h^{-1}(\mca{K})})\cong \Gal(h).$$ 

\noindent {\rm (ii)} For each knot $K$ in $\mca{K}$, we have the following commutative diagram:  
$$\xymatrix{
H_1(\del V_K) \ar[d] \ar[r]^{\cong \ \ \ \ }_{\rm Hur \ \ \ \ }\ar@{}[dr]|\circlearrowleft & \Gal(\wt{\del V_K}/\del V_K) \ar[d] \\
C_{M,\mca{K}} \ar[r]_{\rho_{M,\mca{K}} \ \ \ \ } & \ds \Gal(M,\mca{K})^{\rm ab} ,
}$$
where the vertical maps are induced by the natural inclusions.
\end{thm}

\begin{rem} 
This theorem \violet{refines and} generalizes the main result of \cite{Niibo1}. 
\violet{We have replaced the definition of $P_{M,\mca{K}}$ and hence that of $C_{M,\mca{K}}$. 
In addition, we have removed the assumption that $M$ is an integral homology 3-sphere, that is, $H_1(M)=0$.}   

Since $H_1(M)$ is an analogue of $\Cl(k)$ and $\Cl(k)$ is always finite, 
it may be interesting to restrict the class of $M$ to the rational homology 3-spheres, that is, $H_1(M)$ is finite. 
There are many results on the class number $\#\Cl(k)$ via id\`ele theory, 
and our theory may enable us to consider their analogues. \violet{(See Subsection \ref{appl to genus} also.)} 
\end{rem} 

\violet{The existence of $\rho_{M,\mca{K}}$ is done by Lemma \ref{Pidele}.} 
The compatibility with the local theories (ii) follows from the definition of the map. We 
give the proof of (i) in the following. 

\begin{defn}
We define the {\it unit id\`ele group} of $(M,\mca{K})$ by the meridian group 
\[ U_{M,\mca{K}}:=\((a_K)_K\in I_{M,\mca{K}}\ |\  v_K(a_K)=0{\rm \ in\ }H_1(V_K),{\rm \ for\ all\ }K{\rm \ in\ }\mca{K}\),\]
that is, a subgroup of the ``infinite linear combinations'' $\sum_{K\subset \mca{K}} m_K \mu_K$ $(m_K\in \Z)$ of the meridians of $\mca{K}$ with $\Z$-coefficients.
\end{defn}

\begin{lem}[An improvement of \cite{Niibo1} Proposition 5.7] 
Let $M$ be a closed, oriented, connected 3-manifold equipped with an admissible link $\mca{K}$, and $L$ be a finite link in $\mca{K}$. 
We write $U_{M,\mca{K}}=U_L \oplus U_{{\rm non} L}$, where $U_L$ is the subgroup generated by the meridians of $L$, and $U_{{\rm non} L}:=\Ker ({\rm pr}_L:U_{M,\mca{K}}\surj U_L)$. 
Then we have $I_{M,\mca{K}}/(P_{M,\mca{K}}+U_{{\rm non} L})\cong H_1(X_L)$. 

Especially, if we put $L=\phi$, we have $I_{M,\mca{K}}/(P_{M,\mca{K}}+U_{M,\mca{K}})\cong H_1(M)$. Moreover, if $M$ is an integral homology 3-sphere, we have $I_{M,\mca{K}}=P_{M,\mca{K}}\oplus U_{M,\mca{K}}$.
\end{lem}

\begin{rem}
The assumption in \fgreen{\cite[Proposition 5.7]{Niibo1}} can be paraphrased as follows: 
$H_1(M)$ is torsion free, and the knots $K$ in $\mca{K}$ with non-trivial images in $H_1(M)$ forms its free basis. 
\magenta{We succeeded in removing them.}  
\end{rem}

In the proofs, 
we abbreviate $_{M,\mca{K}}$ by $_M$, and $_{N,h^{-1}(\mca{K})}$ by $_{N}$ for simplicity.

\begin{proof}
For a map $\varphi_L: I_M\surj H_1(X_L)$, we prove $\Ker \varphi_L =P_M+U_{{\rm non} L}$.  
Consider the composite 
$\varphi_L:I_M\surj I_M/P_M=C_M\cong \Gal(M,\mca{K})^{\rm ab}\cong \varprojlim_{L'} H_1(X_{L'})\surj H_1(X_L)$.
For each $L\subset L'\subset \mca{K}$, it factorizes as 
$\varphi_L:I_M\underset{\varphi_{L'}}{\surj} H_1(X_{L'})\underset{pr}{\surj}H_1(X_L)$.  
For the meridian $\mu_K$ of $K$ in $I_M$, the Mayer--Vietoris exact sequence proves 
$N_{L'}:=\Ker({\rm pr}:H_1(X_{L'})\surj H_1(X_L))=\langle\varphi_{L'}(\mu_K)\mid \cyan{K\subset L'- L} \rangle$. 
Hence $\Ker(I_M/P_M\surj H_1(X_L))=U_{{\rm non} L}\mod P_M\cong \varprojlim_{L'} N_{L'}=\Ker(\varprojlim_{L'} H_1(X_{L'})\surj H_1(X_L))$, 
and therefore $\Ker \varphi_L=P_M+U_{{\rm non} L}$. 
\end{proof}

\noindent 
\textbf{proof of Theorem \ref{global CFT for mfd} (i).} 
Since there are isomorphisms
$$C_M/h_*(C_N)\cong (I_M/P_M)/h_*(I_N/P_N)
\cong I_M/(P_M+h_*(I_N)),$$ we consider the natural surjection 
$\varphi': I_M\underset{\varphi_L}{\surj} H_1(X_L)\surj H_1(X_L)/h_*(H_1(Y_L))$. 
Since $\mca{K}$ is very admissible, there is a surjection  
$I_N\surj H_1(Y_L)$, and hence a surjection 
$h_*(I_N)\surj  h_*(H_1(Y_L))$. 
Then, there is the following commutative diagram. 
\[\xymatrix{ \ar@{}[rd]|{\circlearrowright}
h_*(I_N) \ar@{->>}[r]
\ar@{^{(}-_{>}}[d] & h_*(H_1(Y_L)) \ar@{^{(}-_{>}}[d]\\
I_M \ar@{->>}[r]_{\varphi_L} & H_1(X_L)}\]
Since $\Ker \varphi_L=P_M+U_{{\rm non} L}< P_M+h_*(I_N)$,
we have $\Ker \varphi'=\Ker \varphi_L +h_*(I_N)=P_M+h_*(I_N),$
and hence $I_M/(P_M+h_*(I_N))\cong H_1(X_L)/h_*(H_1(Y_L))\cong \Gal(h)$. \ $\square$

\begin{rem}\label{remonK}
In the 
proof \orange{above}, we used the assumption that ``$\mca{K}$ is very admissible'' only to say that $I_{N,h^{-1}(\mca{K})}\to h_*(H_1(Y_L))$ is surjective. Therefore, \orange{in 
the global reciprocity law (Theorem \ref{global CFT for mfd})}, 
we can replace the assumption on $\mca{K}$ by a weaker (and necessary) one: ``for any \orange{finite abelian branched cover} $h:N\to M$, the natural map $I_{N,h^{-1}(\mca{K})}\to h_*(H_1(N))$ is surjective''. 
Especially, if $M$ is an integral homology 3-sphere, 
\orange{then it becomes the empty condition.} 
\end{rem}

\section{The standard topology and the existence theorem 1/2} \label{standard}

In this section, we introduce the standard topology on the id\`ele class group of a 3-manifold, and prove the existence theorem 1/2.\\

Let $M$ be a closed, oriented, connected 3-manifold equipped with a very admissible link $\mca{K}$. 
For each group $\pi_1(\del V_K)\cong H_1(\del V_K)=\langle\mu_K\rangle\oplus \langle\lambda_K \rangle\cong \Z^{\oplus 2}$ of the boundary of \magenta{a} tubular neighborhood of \magenta{each knot $K$ in $\mca{K}$}, we define an analogue of the local topology of $k_\p^{\times}$. Here $\mu_K$ and $\lambda_K$ denote the 
meridian and \magenta{the fixed} longitude of $K$ respectively.  
We first consider the \emph{local norm topology} on $H_1(\del V_K)$, whose open subgroups correspond to the finite abelian covers of $\del V_K$. 
This topology is equal to the \emph{Krull topology}, whose open subgroups are the subgroups of finite indices. 
Then we consider the relative topology on the local inertia group $\langle\mu_K\rangle< H_1(\del V_K)$, 
and re-define the \emph{local topology} on $H_1(\del V_K)$ as the unique topology such that the inclusion $\iota:\langle\mu_K\rangle\inj H_1(\del V_K)$ is open and continuous. 
For this topology, under the identification $\Z\cong \langle\mu_K\rangle\inj H_1(\del V_K)=\langle\mu_K\rangle\oplus \langle\lambda_K\rangle\cong \Z\oplus \Z$, the open subgroup of $H_1(\del V_K)$ has the form $\langle(a,\magenta{0}), (\magenta{b},c)\rangle$ with some $a,b,c\in \Z$, $a\neq 0$. 
Then, the local existence theorem is stated as the 1-1 correspondence between the open subgroups of finite indices and the finite abelian covers. 

With this local topology, $I_{M,\mca{K}}$ is the restricted product with respect to the open subgroups 
$\langle\mu_K\rangle< H_1(\del V_K)$, and $I_{M,\mca{K}}$ equips the \emph{restricted product topology} as follows.
For each finite link $L$ in $\mca{K}$, 
let $G(L):=\prod_{K\subset L}H_1(\del V_K)\times \prod_{K\not \subset L} \langle\mu_K\rangle$, and consider the product topology on $G(L)$.
Then a subgroup $H<I_{M,\mca{K}}$ is open if and only if $H\cap G(L)$ is open for every $L$.

\begin{defn} \cyan{We endow $C_{M,\mca{K}}$ with the quotient topology of the restricted product topology of $I_{M,\mca{K}}$ and call it the \emph{standard topology}.}  
\end{defn}

\cyan{
The restricted product topology on $I_{M,\mca{K}}$ is finer than the relative topology of the product topology of the local topologies, while open subgroups of finite indices coincide. We consider the former in the following.
} 

\cyan{
We study a factorization of $I_{M,\mca{K}}\surj C_{M,\mca{K}}$ which helps us to deal with open subgroups of $C_{M,\mca{K}}$. 
We fix a finite sublink $L_0$ of $\mca{K}$ whose components generate $H_1(M)$. 
For each sublink $L\subset \mca{K}$, we put 
$J_L:=\prod_{K\subset L_0} H_1(\del V_K) \times \prod_{K\subset L-L_0}\langle \mu_K\rangle$. 
Note that $J_{\mca{K}}=G(L_0)$ is an open subgroup of $I_{M,\mca{K}}$. 
}

\cyan{
For finite sublinks $L$ and $L'$ with $L_0\subset L\subset L'\subset \mca{K}$, 
the natural maps form the following commutative diagram. 
\[\xymatrix{ \ar@{}[rd]|{\circlearrowright}
J_{L'} \ar@{->>}[r]
\ar@{->>}[d]_{\rm pr} & H_1(X_{L'}) \ar@{->>}[d]\\
J_L \ar@{->>}[r] & H_1(X_L)}\] 
The natural map $\Ker(J_L'\surj H_1(X_{L'})) \to \Ker(J_L\surj H_1(X_L))$ is surjective. Indeed, let $x\in \Ker(J_L\surj H_1(X_L))$ and let $x$ also denote its image by $J_L\inj J_L\oplus \prod_{K\subset L'-L}\langle \mu_K \rangle =J_{L'}; x\mapsto x+0$. Since $\Ker(H_1(X_L')\surj H_1(X_L))$ is generated by the meridians of $L'-L$, 
There is some $a \in \prod_{K\subset L'-L}\langle \mu_K \rangle$ such that the images of $x$ and $a$ in $H_1(X_{L'})$ coincide. If we put $y=x-a$, then $y \in \Ker(J_L'\surj H_1(X_{L'}))$ and its image in $J_L$ is $x$. 
Since $\{\Ker(J_L\surj H_1(X_L))\}_L$ forms a surjective system and satisfies the Mittag-Leffler condition, 
we have a natural surjection $J_{\mca{K}}=\varprojlim_{{\rm pr},L} J_L\surj C_{M,\mca{K}}$.
}  

\cyan{
For each knot $K'$ in $\mca{K}$ with $K'\not \subset L_0$, we take an element $x_{K'}\in J_{\mca{K}}$ 
satisfying $\lambda_{K'}-x_{K'}\in P_{M,\mca{K}}=\ker \rho_{M,\mca{K}}$. 
Put $Q:=\langle \lambda_{K'}-x_{K'} \mid K'\not\subset L_0\rangle < I_{M,\mca{K}}$.
Then $J_{\mca{K}}\inj I_{M,\mca{K}}\surj C_{M,\mca{K}}$ factors through $I':=I_{M,\mca{K}}/Q \cong (\prod_{K\subset \mca{K}} \Z) \times (\prod_{K\subset L_0} \Z)$. 
}

\cyan{
Let $I'$ be endowed with the quotient topology of the standard topology of $I_{M,\mca{K}}$. Since $J_{\mca{K}}$ is open, the induced group isomorphism $J_{\mca{K}}\congto I'$ is a homeomorphism. 
}

\begin{prop}\label{lemmaQHS}
Let $C_{M, \mca{K}}$ be endowed with the standard topology.
If $M$ is a rational homology 3-sphere, \cyan{then} 
an open subgroup of $C_{M,\mca{K}}$ is of finite index. 
\end{prop}

\begin{proof} 
\cyan{Put} $P'=\Ker(I'\surj C_{M,\mca{K}})$. Then we have $I'/(\prod_K \langle\mu_K\rangle + P') \cong H_1(M)$. 
The assumption on $M$ means that $H_1(M)$ is a finite group, and hence 
$\prod_K \langle\mu_K\rangle + P'< I'$ is of finite index. 
Recall $G(L_0)=J_{\mca{K}}\cong I'$ as topological groups. 
If $V$ is an open subgroup of $I'$, then 
$V\cap \prod_K \langle\mu_K\rangle <\prod_K \langle\mu_K\rangle$ is of finite index. 
Let $U$ be an open subgroup of $C_{M,\mca{K}}$ and let $V$ denote the preimage of $U$ in $I'$. 
Then $V$ is an open subgroup of $I'$ containing $P'$. Therefore $V<I'$ is of finite index, and so is $U<C_{M,\mca{K}}$. 
\end{proof} 

\begin{thm}[The existence theorem 1/2]\label{1/2} 
Let $C_{M, \mca{K}}$ be endowed with the standard topology. 
\magenta{Then the correspondence 
\[(h:N\to M) \mapsto h_*(C_{N,h^{-1}(\mca{K})})\] 
gives a bijection between the set of (isomorphism classes of) 
finite abelian covers of $M$ branched over finite links $L$ in $\mca{K}$ and the set of open subgroups of finite indices of $C_{M,\mca{K}}$ with respect to the standard topology.} 
\end{thm} 

\begin{proof} \cyan{
For each finite link $L$ with} $L_0\subset L\subset \mca{K}$, 
let ${\rm Cov}_L$ denote the set of finite abelian covers $h: N\to M$ branched over sublinks of $L$, 
and let $\mca{O}_L$ denote the set of open subgroups of $C_{M,\mca{K}}$ of finite indices containing $\Ker(C_{M,\mca{K}}\surj H_1(X_L))$. 

Let $U$ be an open subgroup of $C_{M,\mca{K}}$ of finite index and let $V$ denote the preimage of $U$ by $I'\surj C_{M,\mca{K}}$. Since $V$ is an open subgroup of $I'$ of finite index, there is some finite link $L$ with $L_0\subset L\subset \mca{K}$ such that $V$ contains a subgroup $\prod_{K\not \subset L}\langle \mu_K \rangle \times \prod_{K\subset L} a_K\langle \mu_K\rangle$ ($a_K\in \N$) and hence contains   
$\prod_{K\not \subset L}\langle \mu_K \rangle \times \prod_{K\subset L} 0$. 
Therefore, $U$ contains the image of $\prod_{K\not \subset L}\langle \mu_K \rangle \times \prod_{K\subset L} 0$, which coincides \orange{with} $\Ker (C_{M,\mca{K}}\surj H_1(X_L))$. 
Thus the union $\cup_L \mca{O}_L$ coincides \orange{with} the set of all the open subgroups of $C_{M,\mca{K}}$ of finite indices.

Conversely, if $U$ is a subgroup of $C_{M,\mca{K}}$ of finite index containing $\Ker (C_{M,\mca{K}}\surj H_1(X_L))$ for a finite link $L$ with $L_0\subset L \subset \mca{K}$, then $U$ is open. 


For each finite link $L$ with $L_0\subset L \subset \mca{K}$, we have a natural bijection ${\rm Cov}_L\to \mca{O}_L$ by the Galois correspondence. In addition, for each finite links $L$ and $L'$ with $L_0\subset L \subset L' \subset \mca{K}$, the inclusions ${\rm Cov}_L\subset {\rm Cov}_{L'}$ and $\mca{O}_L\subset \mca{O}_L'$ are compatible with the Galois correspondences. 

The union $\cup_L {\rm Cov}_L$ is the set of all the finite abelian covers branched over finite sublinks of $\mca{K}$. 
Since the inductive limit of bijective maps is again bijective, 
\cyan{we obtain the desired bijection.}  
\end{proof}

\section{The norm topology and the existence theorem}
In this section, we introduce the \emph{norm topology} on the id\`ele class group, 
and present the existence theorem.\\ 

Let $M$ be a closed, oriented, connected 3-manifold equipped with a very  admissible link $\mca{K}$ as before. 
In the proofs, 
we use the abbreviations $C_M=C_{M,\mca{K}}$ and $C_N=C_{N,h^{-1}(\mca{K})}$ for a branched cover $h:N\to M$.  

\begin{defn} We define the {\it norm topology} on $C_M$ to be the topology of topological group generated by the family $\mca{V}:=\(h_*(C_{N,h^{-1}(\mca{K})})\)$, where $h:N\to M$ runs through all the finite abelian covers of $M$ branched over finite links in $\mca{K}$. 
\end{defn}

\begin{lem}\label{fsys} 
$\mca{V}$ is a fundamental system of neighborhoods of 0.
\end{lem}

\begin{proof}
For any $V_1, V_2 \in \mca{V}$, it is suffice to prove $\exists V_3\in \mca{V}$ such that $V_3 \subset V_1\cap V_2$. However, we prove $V_3 := V_1\cap V_2\in \mca{V}$.

Let $h_i:N_i\to M$ be a finite abelian cover branched over $L_i$ in $\mca{K}$ for $i=1,2$. 
Let $L:=L_1\cup L_2$, and let $G_L:=\Gal(X_L^{\rm ab}/X_L)$ denote the Galois group of the maximal abelian cover over the exterior $X_L=M\setminus {\rm Int}(V_L)$.
Then, if a cover $h:N\to M$ is unbranched outside $L$, the map $C_M\surj \Gal(h)$ factors through the natural map $\varphi_L:C_M\surj G_L$. 

Let $G_i:=\Ker(G_L\surj \Gal(h_i))<G_L$ for $i=1,2$, and let $G_3:=G_1\cap G_2$. Since $G_3$ is also a subgroup of $G_L$ of finite index, the ordinary Galois theory for branched covers gives a cover $h_3:N_3\to M$ such that $G_3=\Ker(G_L\surj \Gal(h_3))$. (This cover $h_3$ should be called the ``composition cover'' of $h_1$ and $h_2$, because it is  an analogue of the composition field $k_1k_2$ of $k_1$ and $k_2$ in number theory.)  

Now, Theorem \ref{global CFT for mfd} (the global reciprocity law) implies 
$h_{i*}(C_{N_i})=\varphi_L^{-1}(G_i)$ for $i=1,2,3$,
and therefore 
$h_{3*}(C_{N_3})=\varphi_L^{-1}(G_3)=\varphi_L^{-1}(G_1\cap G_2)
=\varphi_L^{-1}(G_1)\cap \varphi_L^{-1}(G_2)
=h_{1*}(C_{N_1})\cap h_{2*}(C_{N_2}).$ 
\end{proof}

\begin{prop} 
Let $C_{M,\mca{K}}$ be endowed with the norm topology. 
A subgroup $V$ of  $C_{M,\mca{K}}$ is open if and only if it is closed and of finite index. 
\end{prop}

\begin{proof} Let $V$ be an open subgroup of $C_M$. 
The coset decomposition of $C_M$ by $V$ proves that $V$ is closed. 
Lemma \ref{fsys} gives a finite abelian branched cover $h:N\to M$ such that $h_{*}(C_N)<V$.
Then Theorem \ref{global CFT for mfd} implies $(h_{*}(C_N):V)(V:C_M)=(h_{*}(C_N):C_M)=\# \Gal (h)$, and hence  $V$ is of finite index. 

The converse is also clear by the coset decomposition.
\end{proof}

Now 
we present the existence theorem for 3-manifolds with respect to both the standard topology and the norm topology, which is the counter part of \textbf{Theorem \ref{global CFT}} (2).

\begin{thm}[The existence theorem
]\label{main thm} \label{existCFT}
Let $M$ be a closed, oriented, connected 3-manifold equipped with a very  admissible link $\mca{K}$. 
Then the correspondence 
\[(h:N\to M) \mapsto h_*(C_{N,h^{-1}(\mca{K})})\]
gives a bijection between the set of (isomorphism classes of) 
finite abelian covers of $M$ branched over finite links $L$ in $\mca{K}$ and the set of open subgroups of finite indices of $C_{M,\mca{K}}$ with respect to the standard topology. Moreover, the latter set coincides with
the set of open subgroups of $C_{M,\mca{K}}$ with respect to the norm topology. 
\end{thm}
  
\begin{proof} 
The former part is done by Theorem \ref{1/2}.
We prove the theorem for the norm topology. 
For a finite abelian cover $h:N\to M$ branched over a finite link in $\mca{K}$, 
the isomorphism $C_M/h_*(C_N)\cong \Gal(h)$ in Theorem \ref{global CFT for mfd} (the global reciprocity law) gives the following bijections. 
\begin{align*}
\(C'\mid h_*(C_N)<C'<C_M\) 
&\LR \(H\mid H<C_M/h_*(C_N)\cong \Gal(h)\) \\
&\LR \({\rm subcovers\ of\ }h\)
\end{align*}
\noindent 
(Injectivity) 
For covers $h_1$ and $h_2$, this bijections proves that 
$h_{1*}(C_{N_1})<h_{2*}(C_{N_2})$ $\underset{\rm iff}{\iff} h_2$ is a subcover of $h_1$, and hence $h_{1*}(C_{N_1})=h_{2*}(C_{N_2}) \underset{\rm iff}{\iff} h_2=h_1$.

\noindent 
(Surjectivity) 
For an open subgroup $C'<C_M$, 
Lemma \ref{fsys} gives a cover $h:N\to M$ such that $h_*(C_N)<C'$, 
and then the above bijection gives a cover $h'$ which corresponds to $C'$.\end{proof}

\begin{cor}
If $M$ is a rational homology 3-sphere, the standard topology and the norm topology on $C_{M,\mca{K}}$ coincide.
\end{cor}

\begin{proof}
By Proposition \ref{lemmaQHS}, 
it follows immediately 
from the existence theorem.
\end{proof}

\section{Remarks} 

\subsection{Norm residue symbols}
In the proof of \red{the existence theorem} 
for number fields (Theorem \ref{global CFT} (2)), the norm residue symbol played an essential role (\cite{Neukirch}).

Let $M$ be a closed, oriented, connected 3-manifold equipped with a very admissible knot set $\mca{K}$.
For a finite abelian cover $h:N\to M$ branched over a finite link $L$ in $\mca{K}$, we define the \emph{norm residue symbol}
 $(\ , h): C_{M,\mca{K}}\surj \Gal(h)$ by the composite of 
$\rho_{M,\mca{K}}:C_{M,\mca{K}}\surj \Gal(M^{\rm ab}/M)$ and $\Gal(M^{\rm ab}/M)\surj \Gal(h)$. 
For this map, we have $\Ker(\ ,h)=h_{*}(C_{N,h^{-1}(\mca{K})})$.

By using the norm residue symbol for 3-manifolds, we can give another proof for the 
existence theorem for 3-manifolds (Theorem \red{\ref{existCFT}}). 

The norm reside symbols are extensions of the Legendre symbol and the linking number (mod 2). 
In number theory, the quadratic reciprocity law was deduced from the global reciprocity law (\cite[Chapter 5]{KKS2}). In a similar way, we can deduce the symmetricity of the linking number (mod 2) from our global reciprocity law (Theorem \red{\ref{global CFT for mfd}} (1)). 


\subsection{The class field axioms}
The axiom of class field theory (\cite{Neukirch}) does not hold 
for our modules: 
Let $M$ be an closed, oriented, connected 3-manifold equipped with a very admissible link $\mca{K}$, $h:N\to M$ a cyclic branched cover of degree $n$ branched over some $L\subset \mca{K}$, and put $G=\Gal(h)$. Then the Tate cohomologies do not necessarily satisfy the following: 
(i) $\wh{H}^0(G,C_{N,h^{-1}(\mca{K})})\cong \Z/n\Z$, 
(ii) $\wh{H}^{1}(G,C_{N,h^{-1}(\mca{K})})=0$, 
(iii) $\wh{H}^i(G,U_{N,h^{-1}(\mca{K})})=0$. 
 
It is recently announced by Mihara that he gave another formulation of idelic class field theory satisfying the axiom of class field theory for 3-manifolds equipped with our very admissible links, compatible with our work, by introducing the notion of \emph{the finite \'etale cohomology} of 1-cocycle sheaves (\cite{Mihara-CFT-fukuoka2016}). 

\subsection{Application to the genus theory} \label{appl to genus}
The genus formula for finite Galois extensions of number fields was given by Furuta with use of id\`ele theory (\cite{Furuta1967}). 
In \cite{Ueki3}, the second author formulated its analogue for finite branched Galois covers over $\Q$HS$^3$'s, gave a parallel proof to the original one by using our id\`ele theory, and generalized Morishita's work (\cite{Morishita2001g}). 

\section*{Acknowledgments}
We would like to express our sincere gratitude to our supervisor Masa nori Morishita for his advice and \fgreen{ceaseless} encouragement. 
We 
would like to thank Tomoki Mihara, who suggested that we should establish the existence theorem, and had good discussions with us. 
\fgreen{We are grateful to Atsushi Yamashita for informing much about the topology of infinite links.}  
In addition, 
we would like to thank Ted Chinberg, 
\orange{Brian Conrad,} 
\red{Ji Feng,} Tetsuya Ito, Yuich Kabaya, 
\fgreen{Dohyeong Kim}, 
\orange{Toshitake Kohno,} 
Makoto Matsumoto, Hitoshi Murakami, 
\orange{Takayuki Oda,} 
\orange{Takayuki Okuda,} 
Adam Sikora, 
\orange{Tomohide Terasoma,} 
\red{Takeshi Tsuji}, 
\orange{Akihiko Yukie,} 
Don Zagier, \red{and the anonymous referees} for giving useful comments, and Kanako Nakajima for checking our English. 
\red{The authors are} partially supported by Grant-in-Aid for JSPS Fellows (
\red{27-7102,} 25-2241).


\bibliographystyle{amsalpha}
\bibliography{niiboueki-icft.tams}

\ \\
\noindent \footnotesize 
Hirofumi Niibo: Faculty of Mathematics, Kyushu University, 744, Motooka, Nishi-ku, Fukuoka, 819-0395, Japan, E-mail: \url{h-niibo@math.kyushu-u.ac.jp} \\[2mm]
Jun Ueki: Graduate School of Mathematical Sciences, The University of Tokyo, 3-8-1 Komaba, Meguro-ku, Tokyo, 153-8914, Japan, E-mail: \url{uekijun46@gmail.com} 
\end{document}

\ \\
Hirofumi Niibo \\
Faculty of Mathematics, Kyushu University \\
744, Motooka, Nishi-ku, Fukuoka, 819-0395, Japan \\
E-mail: \url{h-niibo@math.kyushu-u.ac.jp} \\
\ \\
Jun Ueki \\
Graduate School of Mathematical Sciences, The University of Tokyo \\
3-8-1 Komaba, Meguro-ku, Tokyo, 153-8914, Japan\\
E-mail: \url{uekijun46@gmail.com} 

\end{document}
\newpage

\setcounter{section}{7}

\section{【旧】The norm residue symbols}
In this section, we introduce the norm residue symbol for 3-manifolds, \red{which may be regarded as} an analogue of the norm residue symbol for number fields. We also explain that they generalize the linking number ${\rm lk}(K_1,K_2)$ and the Legendre symbol $\ds \left(\frac{\, p\, }{q}\right)$.

\begin{defn} 
For a finite abelian extension $F/k$, the \emph{norm residue symbol}
 $(\ , F/k): C_k\surj \Gal(F/k)$ is defined as the composite of 
$\rho_k:C_k\surj \Gal(k^{\rm ab}/k)$ and $\Gal(k^{\rm ab}/k)\surj \Gal(F/k)$. 
For this map, we have $\Ker(\ ,F/k)=N_{F/k}(C_F)$. 
\end{defn}

The relation with Legendre's quadratic residue symbol can be seen as follows: 
Let $p$ and $q$ be distinct primes in $k=\Q$, and 
let $F=\Q(\sqrt{q})$ be the quadratic extension of $\Q$ ramified at $q$. 
Then \cite{KKS2} Lemma 5.19 
states the following equivalences:
\begin{alignat*}{2}
\left(\frac{\, q\, }{p}\right)&=1 
&&\iffu (p)=\p_1\p_2 \text{\ with two primes\ }\p_1,\p_2 \text{\ in\ } \mca{O}_F \text{\ (decomposed),}\\ 
\left(\frac{\, q\, }{p}\right)&=-1 
&&\iffu (p) \text{\ is a prime in\ } \mca{O}_F \text{\ (inert).} 
\end{alignat*}
%

On the other hand, 
under the identification $\Gal(F/k)\cong \(\pm 1\)$, 
there are the following equivalences: \\[-3mm]

$((p),F/k)=1$  $\iffu (p) \in N_{F/k}(C_{F/k}) \iffu (p)$ is decomposed in $F/k$. 

Therefore, we have $\ds \left(\frac{\, q\, }{p}\right)=((p),\Q(\sqrt{q})/\Q)$ in $\(\pm1\)$. 

\begin{defn} 
Let $M$ be a 3-manifold equipped with a very admissible knot set $\mca{K}$.
For a finite abelian cover $h:N\to M$ branched over a finite link $L$ in $\mca{K}$, the \emph{norm residue symbol}
 $(\ , h): C_{M,\mca{K}}\surj \Gal(h)$ is defined as the composite of 
$\rho_{M,\mca{K}}:C_{M,\mca{K}}\surj \Gal(M^{\rm ab}/M)$ and $\Gal(M^{\rm ab}/M)\surj \Gal(h)$. 
For this map, we have $\Ker(\ ,h)=h_{*}(C_{N,h^{-1}(\mca{K})})$.
\end{defn}

The relation with the linking number can be seen as follows: 
Let $h_2:N\to M$ be the double cover of $M=S^3$ branched over a knot $K_2$ in a two component link $K_1\sqcup K_2$.
We identify $\Gal(h_2)\cong \Z/2\Z$. 
Then, for a longitude $\lambda_1$ of $K_1$ in $C_{M,\mca{K}}$,  
we have  $(\lambda_1,h_2)={\rm lk}(K_1,K_2)$ (mod 2). Moreover, there are the following equivalences:
\begin{align*}
(\lambda_1,h_2)=0 &\iffu h^{-1}(K_1)=K_1'\sqcup K_1''
\text{\ with knots\ }  K_1', K_1'' \text{\ in\ } N  \text{\ (decomposed),} \\ 
(\lambda_1,h_2)=1 &\iffu h^{-1}(K_1)=\wt{K_1} \text{\ is a knot in\ } N \text{\ (inert).}
\end{align*}

Thus, we have obtained an extension of the dictionary of analogies. 
$$
\begin{tabular}{|c||c|}
\hline 
linking number lk$(K_1,K_2)$ (mod 2) &Legendre symbol $\ds \left(\frac{\, p\, }{q}\right)$\\
\hline 
norm residue symbol $(\ ,h)$& norm residue symbol $(\ ,F/k)$\\ 
\hline 
\end{tabular}
$$

Let $p$ and $q$ be distinct odd primes and $q^*:=(-1)^{\frac{q-1}{2}}q$. 
\red{Then the quadratic reciprocity law $\ds \left(\frac{q^*}{p}\right)=\left(\frac{\, p\, }{q}\right)$ follows from Artin's global reciprocity law (Theorem \ref{global CFT} (1)) (See \cite{KKS2} Chapter 5). 
Similarly, for knots $K_1$ and $K_2$ in $S^3$, 
we can give an alternative proof of ${\rm lk}(K_1,K_2)\equiv {\rm lk}(K_2,K_1)$ $\mod 2$ by using our global reciprocity law (Theorem \ref{global CFT for mfd}).
This fact extends an analogy described in} 
\cite{Morishita2012} Chapter\red{s} 4 and 5. 

In the proof of \red{the existence theorem} 
for number fields (Theorem \ref{global CFT} \red{(2)}), the norm residue symbol plays an essential role (\cite{Neukirch}). 
By using the norm residue symbol for 3-manifolds, we can also give a parallel proof for the 
\red{existence theorem} 
for 3-manifolds (Theorem \red{\ref{existCFT}}), 
although it becomes a little more complicated-looking than our proof in this paper. 


\section{【旧】Axiom of class field theory} 
Finally, we calculate the Tate cohomology of id\`ele class group, and 
compare the result with the axiom of class field theory. 

When a finite cyclic group $G=\langle\sigma\rangle$ of order $n$ acts on a module $A$, 
by definition, the Tate cohomology is calculated as 
$$\wh{H}^0(G,A)=A^G/\Nr A,\ 
\wh{H}^1(G,A)\cong \Ker (\Nr)/(1-\sigma)A.$$  
Here,  $\ds \Nr=\sum_{0\leq i \leq n-1}\sigma^i$ 
denotes the norm map (in a general sense), and the isomorphism class of
 $\wh{H}^i(G,A)$ depends only on $i \mod 2$. 
 
Now the class field axiom for number fields is stated:

\begin{thm}[The axiom of class field theory for number fields \cite{Neukirch}]
Let $F/k$ be a cyclic extension of number fields with degree $n$, and $G=\Gal(F/k)$. Then 
the id\`ele class group $C_F$ satisfies 
$$\wh{H}^0(G,C_F)\cong \Z/n\Z,\ \wh{H}^{1}(G,C_F)=1.$$
\end{thm} 

In addition, we have 

\begin{prop}[\cite{Neukirch}] If  $F/k$ is an unramified cyclic extension, 
the unit id\`ele group $U_F$ satisfies $\wh{H}^i(G,U_F)=1$ for all $i\in \Z$.
\end{prop} 

In the case of number fields, the axiom ensures 
the Artin reciprocity law.  
Therefore, the following question is natural: 
\begin{q}[The axiom of class field theory for 3-manifolds] Let $M$ be an closed, oriented, connected 3-manifold equipped with a very admissible link $\mca{K}$, $h:N\to M$ a cyclic branched cover of degree $n$, and $G=\Gal(h)$. Do the following hold? 
(i) $\wh{H}^0(G,C_{N,h^{-1}(\mca{K})})\cong \Z/n\Z$, 
(ii) $\wh{H}^{1}(G,C_{N,h^{-1}(\mca{K})})=0$. 
\end{q} 

In the following, we will check that the axiom does not \red{necessarily} hold. 
Let $p$ be a prime number and suppose $n=p$. 
For each $G$-module $A$, we write $H^i(G,A)=H^i(A)$ for simplicity. 
For each $A = I, P$, and $C$, we use the abbreviation $A_{N,h^{-1}(\mca{K})}=A_N$.
The exact sequence $0\to P_N\to I_N\to C_N\to 0$ yields the long exact sequence of the Tate cohomologies. 

A natural map called the \emph{transfer} $h^!: C_*(M)\to C_*(N)$ is defined by 
taking a connected component $\wt{c}_1$ of the preimage of an open simplex $c$ and sending $c\mapsto \sum_{g\in G} g\wt{c}_1$. 
(It is also defined by using the language of Serre-fibration on the exterior of branch locus $L$, and by extending to whole $M$ so that it is compatible with the boundary map $\del:C_{i+1}\to C_i$.) 
This map induces the transfers on $Z_*$ and $B_*$, and hence on $H_*$.  
The local behavior of the transfer map is as follows: 
\begin{prop}\label{transfer}
Let $K$ be a knot in $\mca{K}$ and let $\wt{K}=h^{-1}(K)$. \\
(1) If $K$ is branched, then $h^!:H_1(\del V_K)\to H_1(\del V_\wt{K})$ maps $\mu_K\mapsto \mu_\wt{K}$, $\lambda_K\mapsto p\lambda_\wt{K}$. \\
(2) If $K$ is inert, then $h^!:H_1(\del V_K)\to H_1(\del V_\wt{K})$ maps $\mu_K\mapsto p\mu_\wt{K}$, $\lambda_K\mapsto \lambda_\wt{K}$. \\
(3) If $K$ is decomposed, say $\wt{K}=K_1\cup ... \cup K_p$, then  
$h^!:H_1(\del V_K)\to H_1(\del V_\wt{K})$ maps $\mu_K\mapsto \sum_i \mu_{K_i}$, $\lambda_K\mapsto \sum_i \lambda_{K_i}$. 
\end{prop}
The transfers are also induced on 
the id\`ele groups, the principal id\`ele groups, and the id\`ele class groups naturally,  and satisfies 
$\ds \Nr
=h^!\circ h_*$ on each of them. 

\begin{prop} 
The id\`ele group satisfies $\wh{H}^1(G,I_{N,h^{-1}(\mca{K})})=0$. 
\end{prop}

\begin{proof} 
Let $\mca{K}_{\rm BI}$ and $\mca{K}_{\rm D}$ denote 
the sublinks of $h^{-1}(\mca{K})$ which consist of the branched or inert components and the decomposed components respectively, and put $I_i:=\restprod_{K\subset \mca{K}_i}H_1(\del V_K)$ for $i={\rm BI, D}$. Then $I_N=I_{\rm BI}\oplus I_{\rm D}$ as $G$-modules. 

The former part $I_{\rm BI}$ is point-wise fixed by $G$, and $\Nr$ acts on it as multiplication by $p$. Since $I_{\rm BI}$ is torsion-free, 
$\Ker \Nr|_{I_{\rm BI}}=0$, and hence $\wh{H}^1(I_{\rm BI})=0$. 
The latter part satisfies $I_{\rm D}=(\text{meridians})\oplus(\text{longitudes})\cong \Z[G]^\N \oplus \Z[G]^{\oplus \N}$, and it  
is the inverse limit of a surjective system of free $\Z[G]$-modules. 
Since $G$ is finite, we obtain $\wh{H}^*(I_{\rm D})=0$. 
Therefore $\wh{H}^1(I_N)=0$. \end{proof}

\begin{conj} 
(ii) $\wh{H}^1(G,C_{N,h^{-1}(\mca{K})})=0$ holds in general. 
\end{conj}

Then, the exact sequence $\wh{H}^{-1}(I_N)\to \wh{H}^{-1}(C_N)\to \wh{H}^0(P_N)\to \wh{H}^0(I_N)$ yields 
$\wh{H}^1(C_N)\cong \wh{H}^{-1}(C_N)\cong \Ker(\wh{H}^0(P_N)\to \wh{H}^0(I_N)) =\Ker(P_N^G/\Nr(P_N)\to I_N^G/\Nr(I_N))$. 
If $P_N^G/\Nr(P_N)\to I_N^G/\Nr(I_N)$ is injective, then $\wh{H}^1(C_N)=0$ holds. 


Next, let 
$d_1$ and $d_2$ denote the numbers of branched components and inert components of $h^{-1}(\mca{K})$ respectively, where $d_2$ can be infinite, and let $\mca{K}_{\rm BI}$ be as in the proof above. 
\begin{prop} 
The id\`ele group satisfies $\wh{H}^0(G,I_{N,h^{-1}(\mca{K})})\cong 
\restprod_{K' \subset \mca{K}_{\rm BI}}(\Z/p\Z)^2$, 
and $I_{N,h^{-1}(\mca{K})}^G/h^!(I_{M,\mca{K}})\cong (\Z/p\Z)^{d_1+d_2}$. 
\red{Meanwhile}, (i) $\wh{H}^0(G,C_{N,h^{-1}(\mca{K})})\cong \Z/p\Z$ does not \red{necessarily} hold. 
\end{prop} 

\begin{proof} 
Recall the standard isomorphism $I_N=\restprod_{K'\subset h^{-1}(\mca{K})}H_1(\del V_{K'})\cong \restprod_{K'} \Z^2$ defined by the meridians and the fixed longitudes.  
Then $\wh{H}^{\red 0}(I_N)\cong \restprod_{K'\subset \mca{K}_{\rm BI}} (\Z/p\Z)^2$ is clear. 
By Proposition \ref{transfer}, $I_N^G/h^!(I_M)\cong (\Z/p\Z)^{d_1}\oplus (\Z/p\Z)^{d_2}$ 
is also clear, where $(\Z/p\Z)^{d_1}$ and $(\Z/p\Z)^{d_2}$ correspond to
the longitudes of the branched part and the meridians of the inert part respectively. 
Indeed, the transfer $h^!:I_M\to I_N^G$ is bijective on the decomposed part, the longitudes of the inert part and the meridians of the branched part. 

Meanwhile, the maps $I_N^G\to C_N^G$ and $h^!(I_M)\surj h^!(C_M)$ yield a natural map $I_N^G/h^!(I_M)\to C_N^G/h^!(C_M)$. 
In general, this map is clearly non-zero, and especially $C_N^G/h^!(C_M)=0$ does not hold; 
Take the meridian of an inert knot $K$ in $h^{-1}(\mca{K})$ 
which is non-trivial in $H_1(N\setminus K)$, for instance. 

Now we have an exact sequence 
$0\to C_M/h_*(C_N)\to C_N^G/h^!\circ h_*(C_N)\to C_N^G/h^!(C_M)\to 0$, 
an equation $\wh{H}^0(C_N)=C_N^G/h^!\circ h_*(C_N)$, and the isomorphism $G=\Gal(h)\cong C_M/h_*(C_N)$ by Theorem \ref{global CFT for mfd} (the global reciprocity law). 
Thus, $\wh{H}^0(C_N)$ differs from $G\cong \Z/p\Z$ by the term $C_N^G/h^!(C_M)$. 
\end{proof}

Here is another way of calculation: 
an isomorphism 
$C_{M,\mca{K}}\cong \varprojlim_{L\subset \mca{K}} H_1(M\setminus L)$ 
deduces the calculation of $\wh{H}^i(C_{N,h^{-1}(\mca{K})})$ to the case of a finite link $L$  
and the ordinary exact sequence $0\to B_1(X_L)\to Z_1(X_L)\to H_1(X_L)\to 0$. 

In addition, the unit id\`ele group satisfies the following: 

\begin{prop} The Tate cohomology of the unit id\`ele group counts the number of branched or 
 inert components: 
$\wh{H}^0(U_{N,h^{-1}(\mca{K})})\cong (\Z/p\Z)^{d_1+d_2}, 
\wh{H}^1(U_{N,h^{-1}(\mca{K})})=0$. 
Especially, if $h$ is unbranched, it counts the number $d_2$ of inert components. 
\end{prop} 

Thereby, we have checked the counter part of the axiom of class field theory. 
It is interesting that the main theorems of id\`elic class field theory hold nevertheless.

\end{document}